\documentclass[12pt]{amsart}
\usepackage{graphicx} 

\usepackage{soul} 
\usepackage{amscd,amssymb,mathtools}
\usepackage[arrow,matrix,graph,frame,poly,arc,tips]{xy}
\usepackage[dvipsnames]{xcolor}
\usepackage{calrsfs}
\usepackage[labelfont=rm]{subcaption}
\usepackage{tikz-cd} 
\usepackage{tikz}
\usetikzlibrary{decorations.markings}
\usetikzlibrary{shapes}
\usetikzlibrary{backgrounds}
\usetikzlibrary{calc}
\usetikzlibrary{shapes.misc, positioning}
\usepackage{mdwlist}
\usepackage{xfrac}
\DeclareCaptionSubType*{figure}

\usepackage{multirow}
\usepackage{enumerate}
\usepackage{verbatim}
\usepackage{comment}
\usepackage{todonotes}

\DeclarePairedDelimiter{\form}{\langle}{\rangle}
\DeclarePairedDelimiter{\abs}{\lvert}{\rvert}

    \newcommand\ba{\begin{align*}}
    \newcommand\ea{\end{align*}}
    \newcommand\be{\begin{enumerate}}
    \newcommand\ee{\end{enumerate}}
    \newcommand\bpf{\begin{proof}}
    \newcommand\epf{\end{proof}}
    \newcommand\bpp{\begin{prop}}
    \newcommand\epp{\end{prop}}
    \newcommand\bpb{\begin{prob}}
    \newcommand\epb{\end{prob}}
    \newcommand\bd{\begin{defn}}
    \newcommand\ed{\end{defn}}
    \newcommand\bh{\begin{hint}}
    \newcommand\eh{\end{hint}}

    \newcommand\bZ{\mathbb{Z}}
    \newcommand\Z{\mathbb{Z}}

    \newcommand\cay{\operatorname{Cayley}}

\DeclareMathOperator\radconv{\rho}

    \newcommand\rk{\operatorname{rk}}

    \DeclareMathOperator\Aut{Aut}

    \def\thetitle{Small  growth rates of free groups}
    \def\theauthors{{Koji Fujiwara, Sang-hyun Kim, Ryokichi Tanaka}}
    \usepackage{hyperref}
    \hypersetup{
    	colorlinks=false,
    	plainpages,
    	urlcolor=black,
    	linkcolor=black,
    	pdftitle=   \thetitle,
    	pdfauthor=  {\theauthors}
    }

    \theoremstyle{plain}
    
    \newtheorem{thm}{Theorem}[section]
    \newtheorem{lem}[thm]{Lemma}
    \newtheorem{cor}[thm]{Corollary}
    \newtheorem{prop}[thm]{Proposition}
    
    \newtheorem{que}[thm]{Question}
    
    \newtheorem*{claim*}{Claim}

    \theoremstyle{remark}
    \newtheorem{exmp}[thm]{Example}
    \newtheorem{rem}[thm]{Remark}

    \theoremstyle{definition}
    \newtheorem{defn}[thm]{Definition}
    \newtheorem{prob}{Problem}[section]

\begin{document}
    \title\thetitle
    \date{\today}

    
    \author[K. Fujiwara]{Koji Fujiwara}
    \address{Okinawa Institute of Science and Technology, Okinawa, 904-0495, Japan; \\
    KUIAS and RIMS, Kyoto University, Kyoto, 606-8502, Japan}
    \email{koji.fujiwara@oist.jp, fujiwara.koji.4n@kyoto-u.jp}    
    \thanks{K.F.\ is supported by JSPS Grant-in-Aid for Scientific Research 25H00588.}
    \urladdr{}
    
    \author[S. Kim]{Sang-hyun Kim}
    \address{School of Mathematics, Korea Institute for Advanced Study (KIAS), Seoul, 02455, Korea}
    \email{skim.math@gmail.com}
    \thanks{S.K.\ is supported by Mid-Career Researcher Program (RS-2023-00278510) through the National Research Foundation funded by the government of Korea, by KIAS Individual Grant (MG073602) at Korea Institute for Advanced Study,
    and by KIAS--KAIST Joint Research Group Grant.}
    \urladdr{https://kimsh.kr}

    \author[R. Tanaka]{Ryokichi Tanaka}
    \address{Department of Mathematics, 
    Kyoto University, Kyoto, 606-8502, Japan}
    \email{rtanaka@math.kyoto-u.ac.jp}
    \thanks{R.T.\ is supported by 
JSPS Grant-in-Aid for Scientific Research 24K06711.}
    \urladdr{}

\begin{abstract}
We introduce the notion of a Magnus marking for a finite generating set of a group and prove a strong expansion property for such markings. Using this together with a method based on spherical growth series, we determine the second and third smallest growth rates of the rank-$d$ free group for every $d\ge 2$, and characterize all finite generating sets that realize these growth rates, up to natural equivalences.

As applications of the expansion estimate, we obtain a lower bound for the growth rate associated with a one-relator presentation and prove that the minimal growth rate $e(\Gamma_g)$ of the fundamental group $\Gamma_g$ of a closed orientable surface of genus $g\ge 2$ is strictly greater than $4g-1-\frac{2}{2g-1}$. We also show that $e(\Gamma_g)$ is realized by a generating set with $2g$ elements.
\end{abstract}

    \maketitle
    
\setcounter{tocdepth}{1}
    

\section{Introduction}\label{sec:intro}
Let $G$ be a finitely generated group with a finite generating set $T$. We denote by
$B_{G,T}(n)=B_T(n)$ the set of elements of $G$ whose word lengths are at most $n$
with respect to $T$. The growth rate of $G$ with respect to $T$ is
\[
 e(G,T):=\lim_{n\to\infty}|B_T(n)|^{1/n}.
\]
We say that $G$ has \emph{exponential growth} if $e(G,T)>1$ for some (equivalently,
for every) finite generating set $T$. We define the \emph{minimal growth rate}, also
called the \emph{uniform exponential growth rate}, by
\[
 e(G):=\inf_T e(G,T),
\]
where $T$ runs over all finite generating sets of $G$. If $e(G)>1$, then $G$ is said
to have \emph{uniform exponential growth}. Following \cite{FS}, we consider the set of all
growth rates
\[
 \xi(G):=\{e(G,T)\mid T\text{ is a finite generating set of }G\}.
\]

For a non-elementary word-hyperbolic group $G$, it is known that $e(G)>1$ \cite{K}.
Moreover, the first author and Sela proved in \cite{FS} that $\xi(G)$ is well-ordered in
$\mathbb R$. In particular, the infimum $e(G)$ is attained, answering a question of
de la Harpe \cite{dlHarpe2000,dlH}. They also showed that the order type of $\xi(G)$ is at least
$\omega^\omega$, where $\omega$ denotes the order type of the natural numbers. We
refer to \cite{L} for an exposition of these results.

The well-ordering result naturally raises the problem of determining the first several
growth rates for important classes of hyperbolic groups. Determining these values is difficult even for the most basic examples. In this paper we study this problem for free groups and
surface groups. For free groups, we determine the second and third smallest growth
rates, and also determine all finite generating sets that realize these values, up to
natural equivalences. For surface groups, the problem of determining the minimal 
growth rate remains open, but we obtain a new lower bound and show that the minimum
is realized by a generating set of minimal cardinality.

Our equality analysis reveals a form of growth spectral rigidity. Although growth rates are usually viewed as soft numerical invariants, in the extremal cases studied here they encode strong algebraic information, determining generating sets up to natural equivalences and constraining associated free subgroups.
We introduce our results and elaborate on these points below. 

\subsection{The first growth rates of free groups}
Let $F_d$ be the free group of rank $d\ge 2$. It is known that the order type of
$\xi(F_d)$ is $\omega^\omega$ \cite{FS}. For a positive integer $n$, we write $e_n(F_d)$
for the $n$-th smallest growth rate. 
The minimal growth rate is 
\[
 e_1(F_d)=2d-1,
\]
and it is realized by a free basis \cite{dlHarpe2000,dlH} (see also \cite{S}). We determine the next two
smallest growth rates.
Throughout,
$\{a_1,\ldots,a_d\}$ denotes a free basis of $F_d$.

\begin{thm}[Small growth rates of $F_d$]\label{thm:main}
Suppose $d\ge 2$.
\begin{enumerate}[(1)]
\item\label{eq:main1} The second smallest growth rate of $F_d$ is
\[
 e_2(F_d)=d-1+\sqrt{d^2+4d-4}.
\]
It is realized by a generating set 
\[
 \{a_1,a_1^2,a_2,a_3,\ldots,a_d\}.
\]
\item\label{eq:main2} The third smallest growth rate of $F_d$ is
\[
 e_3(F_d)=\frac{2d-1+\sqrt{4d^2+12d-15}}{2}.
\]
It is realized by generating sets
\[
 \{a_1,a_1^3,a_2,a_3,\ldots,a_d\}
 \quad\text{and}\quad
 \{a_1,a_1a_2,a_2,a_3,\ldots,a_d\}.
\]
\end{enumerate}
\end{thm}

For example, the three smallest growth rates of $F_2$ are
\[
 3,\qquad 1+2\sqrt2,\qquad 4.
\]
The growth rates of the generating sets displayed in Theorem~\ref{thm:main} can be
computed explicitly. 
The main difficulty is to show that these are indeed the next two values in the growth spectrum.
An arbitrary generating set of $F_d$
need not resemble either a free basis or the generating sets appearing in the theorem.
 Thus, a main
part of the proof is to obtain lower bounds that apply uniformly to arbitrary finite
generating sets. 
For this purpose, we introduce Magnus markings, defined below, as one of the main tools.

We also determine the equality cases. Let $T,T'$ be two finite generating sets of a
group $G$. We say that they are \emph{essentially equivalent} if one can be obtained
from the other by a sequence of the following operations: adding or removing the
identity, adding or removing the inverse of an element in the set, and 
replacing an element by its inverse. These operations do not change the growth rate.

\begin{thm}[Equality conditions]\label{thm:main.eq}
Suppose $d\ge 2$.
\begin{enumerate}
\item We have $e(F_d,T)=e_2(F_d)$ if and only if $T$ is, up to $\Aut(F_d)$,
essentially equivalent to the generating set 
\[
 \{a_1,a_1^2,a_2,a_3,\ldots,a_d\}.
\]
\item We have $e(F_d,T)=e_3(F_d)$ if and only if $T$ is, up to $\Aut(F_d)$,
essentially equivalent to one of the following two generating sets 
\[
 \{a_1,a_1^3,a_2,a_3,\ldots,a_d\},
 \qquad
 \{a_1,a_1a_2,a_2,a_3,\ldots,a_d\}.
\]
\end{enumerate}
\end{thm}

Thus, the first two gaps in the growth spectrum are determined together with all generating sets that realize $e_2(F_d)$ and $e_3(F_d)$. 
Although the conclusions of Theorem \ref{thm:main.eq} may look natural, 
their proof is 
subtle. 
The two different equality cases for $e_3(F_d)$ reflect the two main mechanisms in the proof, which we describe in Section~\ref{sec:outline}.

\subsection{Magnus markings and growth estimates}
One of the main tools of the paper comes from Magnus' Freiheitssatz. Consider a
one-relator presentation
\[
 G=\form{ T\mid r},
\]
where $r$ is cyclically reduced and every element of $T$, or its inverse, appears in
$r$. The Freiheitssatz \cite{Magnus} states that every proper subset of $T$ freely generates a
free subgroup of $G$. Motivated by this property, we make the following definition.
A finite generating set $T$ of a group $G$ is a \emph{Magnus marking} if every proper
subset of $T$ freely generates a free subgroup.

If $T=\{y_0,\ldots,y_t\}$ is a Magnus marking, then deleting any one generator gives a free basis of a rank-$t$ subgroup. In the Cayley graph, these $t+1$ subgroups give rise to $t+1$ spanning forests. Using all these forests simultaneously gives a strong lower bound on the vertex boundary of finite subsets of the Cayley graph.

\begin{lem}[Expansion property; see Proposition~\ref{lem:nonsplit} for a stronger version]\label{lem:nonsplit_intro}
Let $T$ be a Magnus marking of a group $G$, and put $t=|T|-1\ge 2$. Then for every
nonempty finite subset $A\subset G$,
\[
 |A(T\cup T^{-1})\setminus A|
 \ge 2\left(t-\frac1t\right)|A|+\frac{2(t+1)}{t}.
\]
In particular,
\begin{equation}\label{eq:growth}
 e(G,T)\ge 1+2t-\frac2t.
\end{equation}
\end{lem}

The lower bound on the growth rate follows by applying the first inequality to
$A=B_T(n)$:
\[
 |B_T(n+1)|\ge \left(1+2t-\frac2t\right)|B_T(n)|.
\]

Proposition~\ref{lem:nonsplit} refines this estimate by taking into account how many cosets of these free subgroups meet $A$. This refinement is used in the equality analysis.

\begin{exmp}[Rank two]
For $F_2=\form{ a,b}$, the generating set $T=\{a,ab,b\}$ is a Magnus marking,
but $\{a,a^2,b\}$ is not. Here $t=2$. Applying Lemma~\ref{lem:nonsplit_intro}
to the balls $A=B_T(n)$, we obtain
\[
 |B_T(n+1)|\ge 4|B_T(n)|.
\]
Hence $e(F_2,T)\ge4$. A direct computation gives $e(F_2,T)=4$. On the other hand,
\[
 e(F_2,\{a,a^2,b\})=1+2\sqrt2.
\]
These are precisely the second and third growth rates of $F_2$.
\end{exmp}

\subsection{One-relator groups and surface groups}
The expansion estimate in Lemma~\ref{lem:nonsplit_intro} has applications beyond free groups. First, the Freiheitssatz
immediately gives the following bound for one-relator groups.

\begin{cor}\label{cor:Freiheitssatz}
Let
\[
 G=\form{ T\mid r}
\]
be a one-relator presentation, where $r$ is cyclically reduced and every generator in
$T$, or its inverse, appears in $r$. If $|T|\ge 3$, then
\[
 e(G,T)\ge 2|T|-1-\frac{2}{|T|-1}.
\]
\end{cor}
\begin{proof}
By Magnus' Freiheitssatz, every proper subset of $T$ freely generates a free subgroup.
The conclusion follows from Lemma~\ref{lem:nonsplit_intro}.
\end{proof}

Under the same hypothesis, Gromov obtained the lower bound $2|T|-3$ using the
Freiheitssatz \cite[p.~282, after Conjecture 5.14]{Gromov}. Corollary~\ref{cor:Freiheitssatz}
improves this bound.

If, in addition, $G$ is word-hyperbolic and not virtually free, then the inequality is 
strict; see Corollary \ref{cor:one-relator-strict}.

We next consider the fundamental group $\Gamma_g$ of a closed orientable surface of
genus $g\ge 2$, with the standard presentation
\[
 \Gamma_g=\form{ a_1,b_1,\ldots,a_g,b_g\mid
 [a_1,b_1]\cdots[a_g,b_g]},
\]
and the standard generating set
\[
 T_g=\{a_1,b_1,\ldots,a_g,b_g\}.
\]
Since $T_g$ is a Magnus marking, Corollary~\ref{cor:Freiheitssatz} gives
\[
 e(\Gamma_g,T_g)\ge 4g-1-\frac{2}{2g-1}.
\]
The problem of determining $e(\Gamma_g)$ is open. In particular, it is not known
whether $T_g$ realizes the minimal growth rate \cite[Open problem VII.16]{dlHarpe2000}. A more
refined argument gives the following strict lower bound for the minimum itself.

\begin{thm}\label{thm:surface}
If $g\ge 2$, then
\[
 e(\Gamma_g)>4g-1-\frac{2}{2g-1}.
\]
\end{thm}

Since $\Gamma_g$ is generated by $2g$ elements, we also have
\[
 e(\Gamma_g)\le 4g-1=e(F_{2g}).
\]
Thus Theorem~\ref{thm:surface} places the minimal growth rate in the interval
\[
 4g-1-\frac{2}{2g-1}<e(\Gamma_g)\le 4g-1.
\]
The lower bound improves the previously known bound $4g-3$ \cite[Proposition VII.15]{dlHarpe2000}
(see also \cite{GdlH}).
Our proof is another application of the Magnus-marking method. If
$e(\Gamma_g,T)<4g-1$, then by Lemma~\ref{lem:surface},
$T\cup T^{-1}$ contains a $2g$-element generating set $X$ of $\Gamma_g$.
Every proper subset of $X$ freely generates the subgroup it generates, 
hence $X$ is a Magnus marking. Lemma~\ref{lem:nonsplit_intro} gives the non-strict
lower bound in Theorem~\ref{thm:surface}, and
Lemma~\ref{lem:proper-strong-intro} makes the inequality strict.

We also obtain the following consequence.
\begin{cor}[see Corollary~\ref{cor:min-gen}]\label{cor:min-gen-intro}
For every $g\ge 2$, the minimal growth rate $e(\Gamma_g)$ is realized by a generating
set with $2g$ elements.
\end{cor}
Thus, although we do not determine the minimizing generating set for a surface group,
one may restrict the problem to the generating sets of minimal cardinality. Together with
Theorem~\ref{thm:surface}, this gives further evidence that the standard generating
set may realize $e(\Gamma_g)$.

The same reduction also gives a short proof of a theorem of Los \cite{Los} for geometric
generating sets: among such generating sets, the standard one realizes the minimal
growth rate, see Corollary~\ref{cor:geo-min}. Our proof does not use the symbolic dynamics appearing in \cite{Los}.

\subsection{Proof strategy and key propositions}\label{sec:outline}

We now outline the proofs of Theorems~\ref{thm:main} and~\ref{thm:main.eq}.
We first introduce the $Q$-function, which will be used to compare growth
rates of free products for various generating sets.

Let
\[
P_{G,T}=\sum_{g\in G}X^{|g|_T}
\]
be the spherical growth series, and let $\rho_{G,T}$ be its radius of
convergence. For $0\le x<\rho_{G,T}$, define
\[
Q_{G,T}(x):=1-\frac{1}{P_{G,T}(x)}.
\]

If $G=G_1*\cdots *G_k$ and $T=\bigsqcup_{i=1}^k T_i$, where $T_i$
is a generating set of $G_i$, then Corollary~\ref{cor:spherical-free-finite}
gives
\[
Q_{G,T}(x)=\sum_{i=1}^k Q_{G_i,T_i}(x),
\qquad 0\le x<\rho_{G,T}.
\]
The sum on the right is defined on the larger interval
\[
0\le x<\rho_{\min}:=\min_{1 \le i\le k}\rho_{G_i,T_i},
\]
and $\rho_{G,T}$ is the unique point in this interval satisfying
\[
\sum_{i=1}^k Q_{G_i,T_i}(x)=1.
\]

This description of the growth rate in terms of the $Q$-function will be
used repeatedly in the proofs of Theorems~\ref{thm:main} and
\ref{thm:main.eq}.

\bigskip

A finite generating set of $F_d$ can be reduced to one of the following
two cases.

\begin{enumerate}[(A)]
\item\label{p:cyc} \emph{Cyclically splitting.} There is a splitting
$F_d=\mathbb Z*\cdots*\mathbb Z$ such that each generator belongs to a free factor.
\item\label{p:noncyc} \emph{Cyclically non-splitting factor.} There is a subset $U$
of cardinality $t+1$ generating a rank-$t$ free group for some $2\le t\le d$, with
no commuting pair.
\end{enumerate}

We use the following notation for the generating sets  in the two cases.
For relatively prime integers $1\le p<q$, set
\[
\alpha_d(p,q):=
e\bigl(F_d,\{a_1^p,a_1^q,a_2,\ldots,a_d\}\bigr).
\]
Also, set
\[
\Delta_d:=\{a_1,a_2,a_1a_2,a_3,\ldots,a_d\},
\qquad
\beta_d:=e(F_d,\Delta_d).
\]

In case~\eqref{p:cyc}, the $Q$-function reduces the growth estimate to
computations for generating sets of $\mathbb Z$. These computations give the
following proposition.

\begin{prop}\label{prop:cyclic}
For all $d\ge2$, the following hold.
\begin{enumerate}
\item\label{p:12} $\alpha_d(1,2)$ is the positive root of
\[
 x^2-(2d-2)x-6d+5=0;
\]
in particular,
\[
 \alpha_d(1,2)=d-1+\sqrt{d^2+4d-4}.
\]
\item\label{p:13} $\alpha_d(1,3)$ is the positive root of
\[
 x^2-(2d-1)x-4d+4=0;
\]
in particular,
\[
 \alpha_d(1,3)=\frac{2d-1+\sqrt{4d^2+12d-15}}2.
\]
\item\label{p:pq} For every pair of relatively prime integers $1\le p<q$ with
$(p,q)\ne(1,2)$,
\[
 \alpha_d(1,2)<\alpha_d(1,3)\le \alpha_d(p,q).
\]
\item\label{p:b} $\beta_d=\alpha_d(1,3)$.
\end{enumerate}
\end{prop}

Thus, among the values $\alpha_d(p,q)$, the smallest is
$\alpha_d(1,2)$, and the next smallest is $\alpha_d(1,3)$.

In case~\eqref{p:noncyc}, one finds a Magnus marking of a free factor.
The expansion estimate gives a lower bound for this factor, and the
$Q$-function handles the remaining free factors. This gives the following
proposition.

\begin{prop}\label{prop:noncyc}
Let $d\ge2$. Let $U$ be a finite generating set of $F_d$ with cardinality $d+1$
that does not contain a commuting pair. Then
\[
 e(F_d,U)\ge \beta_d.
\]
\end{prop}

The bound is sharp: it is realized by
 $\Delta_d$.
Thus, the lower bounds from cases~\eqref{p:cyc} and~\eqref{p:noncyc}
coincide at
\[
\beta_d=\alpha_d(1,3).
\]
Propositions~\ref{prop:cyclic} and~\ref{prop:noncyc} are the two main
ingredients in the proof of Theorem~\ref{thm:main}.

The equality statement requires additional work. In the non-splitting case,
Proposition~\ref{lem:1-2-12} analyzes rank-two Magnus markings of growth $4$
and shows that they are essentially equivalent to $\{a,b,ab\}$.
Lemma~\ref{lem:noncyc-eq_new} extends this rigidity to an arbitrary rank, while
Lemma~\ref{lem:no-extra-triangle} controls the effect of adjoining further
generators. Together these results give the following proposition.

\begin{prop}\label{prop:noncyc.eq}
Let $d\ge2$, and let $V$ be a finite generating set of $F_d$ such that $|V|=d+1$
and $V$ contains no commuting pair. If a finite set $T\subset F_d$ satisfies
$V\subset T$ and
\[
 e(F_d,T)=e(F_d,V)=\beta_d,
\]
then $T$ is, up to $\Aut(F_d)$, essentially equivalent to
\[
 \Delta_d=\{a_1,a_2,a_1a_2,a_3,\ldots,a_d\}.
\]
\end{prop}
A key ingredient in the proof of Proposition~\ref{prop:noncyc.eq} is a
strict form of the Magnus-marking expansion estimate. The same strictness
is also used in the proof of Theorem~\ref{thm:surface}. We record the
general statement here.

\begin{lem}[Lemma~\ref{lem:proper-strong}]\label{lem:proper-strong-intro}
Suppose $T$ is a Magnus marking of a word-hyperbolic group $G$ such that
$t:=|T|-1\ge2$. If
\[
 H:=\form{ T\setminus\{y\}}
\]
is an infinite-index quasi-convex subgroup of $G$ for some $y\in T$, then
\[
 e(G,T)>1+2t-\frac{2}{t}. 
\]
\end{lem}

\subsection{Organization of the paper}

Section~2 develops the spherical growth series and the $Q$-function method,
including the free product formula and the cyclic computations, and proves
Proposition~\ref{prop:cyclic}. Section~3 proves the expansion estimate for
Magnus markings and Proposition~\ref{prop:noncyc}. Section~4 combines these
results to prove Theorem~\ref{thm:main}. Section~5 proves
Proposition~\ref{prop:noncyc.eq}, Lemma~\ref{lem:proper-strong-intro}, and
Theorem~\ref{thm:main.eq}. Section~6 proves Theorem~\ref{thm:surface} and
Corollary~\ref{cor:min-gen-intro}, and treats geometric generating sets.
Section~7 discusses several questions suggested by the results.
For the reader's convenience, the Appendix includes a proof of Theorem~\ref{lem:dfw} by Dahmani, Futer, and Wise~\cite{DFW2019}, which is used in Section~5.

\subsection{The use of AI in this paper}
Lemma~\ref{lem:nonsplit_intro} is one of the essential steps in the proof of
Theorem~\ref{thm:main}. Some key ideas for the proof of this lemma and for the
equality analysis in Theorem~\ref{thm:main.eq} were developed through extensive
discussions with the AI model ChatGPT 5.6 Sol Pro by OpenAI. The authors then refined
and generalized the arguments and checked the details and computations. The paper,
including its computations and references, was written by the authors.

\subsection*{Acknowledgments}
The authors would like to thank Martin Bridson, Mark Lackenby, Zlil Sela, and Yasushi
Yamashita for helpful discussions and suggestions. The first author would like to
thank the Mathematical Institute of the University of Oxford for its hospitality.
While preparing this manuscript, we learned that Carl-Fredrik Nyberg Brodda had
independently determined $e_2(F_2)$ and had been working on $e_2(F_d)$. We are
grateful to him for sharing his manuscript with us \cite{NB}.

\section{Spherical growth series}\label{sec:spherical}
Let $G$ be a group with a finite generating set $T$. 
We denote by $|g|_T$ the word length of $g\in G$ with respect to the symmetrized generating set 
$$T^{\pm}:=T\cup T^{-1}.$$

Define the Cayley graph $\cay(G,T)$ for $(G, T^{\pm})$ by multiplying a generator on the right. 
We denote by $S_T(n)$ the radius-$n$ sphere centered at the identity in $\cay(G,T)$.


We define the {\it spherical growth series} of $(G,T)$ as the formal power series 
$$P_{G,T}:=\sum_{g\in G} X^{|g|_T} = 1 + \sum_{n\ge1} |S_T(n)|X^n.$$
For the power series $P_{G, T}$, we denote by $\radconv_{G, T}$ the radius of convergence. 
For $x\in[0,\rho_{G,T})$, we denote by $P_{G,T}(x)$ the value of $P_{G,T}$ at $x$.

It will be convenient for us to consider the function
$$Q_{G,T}(x):=1-\frac1{P_{G,T}(x)}$$
defined for $x\in[0,\rho_{G,T})$.

We record several basic properties of $\rho_{G,T}$, and 
functions $P_{G,T}(x)$ and $Q_{G,T}(x)$. 

\begin{lem}\label{lem:spherical-basic}
Let $G$ be an infinite group with a finite generating set $T$.
Then,
\begin{enumerate}
\item\label{p:free-b-prod1} 
$1 \ge \radconv_{G,T}>0$ and $e(G,T)=1/\radconv_{G,T}$.
\item\label{p:critical} On the interval $[0,\rho_{G,T})$, the functions $P_{G,T}(x)$ and $Q_{G,T}(x)$ are strictly increasing, with
$$ P_{G,T}(x) \ge 1, \, \, Q_{G,T}(x)\in[0,1).$$
\item\label{p:hyp}
We have 
$$
\lim_{x\nearrow \rho_{G,T}} P_{G,T}(x)=\infty, \, \, 
\lim_{x\nearrow \rho_{G,T}} Q_{G,T}(x)=1.$$
\end{enumerate}
\end{lem}

\begin{proof}
Part \eqref{p:free-b-prod1} follows since $\rho_{G, T}=\lim_{n \to \infty}|S_T(n)|^{-1/n}$. 
Part \eqref{p:critical} follows from the definitions of $P_{G,T}$ and $Q_{G,T}$.

We show Part \eqref{p:hyp}.
Let $b_n=|B_T(n)|$ for brevity. Since $b_{m+n} \le b_mb_n$, we have 
$$e(G,T)=\inf_{1 \le n} b_n^{1/n}.$$
Set $e=e(G,T)$ and $\rho=\rho_{G,T}=1/e.$

Suppose first that $e>1$, so that $\rho <1$. 
Then, $b_n \ge e^n$, and hence
$$\sum_{0 \le n} b_n \rho^n =\infty.$$
Since the coefficients are non-negative, 
$$\lim _{x \nearrow \rho} \sum_{0 \le n} b_n x^n =\infty.$$

On the other hand, since $b_n=\sum_{k=0}^n |S_T(k)|$, 
we have the following identity of formal power series 
$$(1-X) \sum_{0 \le n} b_n X^n = P_{G,T}.$$
Hence, for $0 \le x <\rho$,
\[
(1-x)\sum_{n\ge 0} b_n x^n=P_{G,T}(x).
\]
Since $\rho <1$, it follows 
$$\lim_{x \nearrow \rho} P_{G,T}(x) =\infty.$$

Suppose now that $e=1$.
Then $\rho=1$.
Since $G$ is infinite, $P_{G,T}$ has infinitely many non-zero 
coefficients, and therefore,
$$\lim_{x \nearrow 1} P_{G,T}(x)=\infty.$$

Finally, since $Q_{G,T}(x)=1-1/P_{G,T}(x)$, in either case, 
$$\lim_{x\nearrow \rho} Q_{G,T}(x)=1,$$
as required.
\end{proof}

We will repeatedly use the following criterion to find lower bounds for growth rates of free products.

\begin{prop}[The $Q$-formula for free products]\label{lem:spherical-free}
Let $G$ and $H$ be infinite groups with finite generating sets $T$ and $U$, respectively. Set 
$$\rho_{\min}:=\min (\rho_{G,T},\rho_{H,U}).$$ 
Then the following holds.
\begin{enumerate}
\item\label{p:free-prod1}
We have that
$$\rho_{G\ast H,T\sqcup U}
< \rho_{\min}.$$
\item\label{p:free-prod2} 
The following equalities hold as formal power series:
$$
1-\frac{1}{P_{G\ast H,T\sqcup U}}
=
1-\frac{1}{P_{G,T}}
+1-
\frac{1}{P_{H,U}}.$$

Consequently,  for every $0\le x<\rho_{G\ast H,T\sqcup U}$, we have
$$Q_{G\ast H,T\sqcup U}(x)=Q_{G,T}(x)+Q_{H,U}(x).$$

\item\label{p:phi-strict}
There exists  a unique $x_0\in (0,\rho_{\min})$
such that 
 $$Q_{G,T}(x_0)+Q_{H,U}(x_0)=1.$$
 Moreover, we have
 $$x_0=\rho_{G\ast H,T\sqcup U}.
 $$
 Consequently, 
 $$e(G\ast H,T\sqcup U)=\frac{1}{x_0}>\max(e(G,T),e(H,U)).$$
\end{enumerate}    
\end{prop}

\begin{proof}
Since $G$ and $H$ embed isometrically in $G*H$ with respect to 
these generating sets, 
$$\frac{1}{\radconv_{G\ast H,T\sqcup U}} = e(G\ast H,T\sqcup U)\ge\max (e(G,T),e(H,U))=\frac{1}{\rho_{\min}}.$$
Hence $\rho_{\min} \ge \radconv_{G\ast H,T\sqcup U}$. The strict inequality will be shown in the proof of Part \eqref{p:phi-strict}.

For the proof of Part \eqref{p:free-prod2}, see \cite[4.\ Proposition (ii)]{dlHarpe2000}.

We show Part (\ref{p:phi-strict}).
Define a strictly increasing  function
$$\Phi(x)=\Phi_{(G,T),(H,U)}(x):=Q_{G,T}(x)+Q_{H,U}(x), \qquad x\in[0,\rho_{\min}).$$

Without loss of generality, suppose that $\rho_{\min}=\rho_{G,T}$.
By Lemma \ref{lem:spherical-basic} \eqref{p:hyp},
we have $\lim_{x\nearrow \rho_{\min}}Q_{G,T}(x)=1$. 
Also, $\lim_{x\nearrow \rho_{\min}}Q_{H,U}(x)>0$
since $Q_{H,U}(x)$ is positive and monotone. 
Hence, 
$$\lim_{x\nearrow \rho_{\min}}\Phi(x)
=\lim_{x\nearrow \rho_{\min}}Q_{G,T}(x)
+\lim_{x\nearrow \rho_{\min}}Q_{H,U}(x)
= 1+\lim_{x\nearrow \rho_{\min}}Q_{H,U}(x)>1.$$

Since $\Phi(0)=0$ and $\Phi$ is continuous and strictly increasing on
$[0,\rho_{\min})$, there exists a unique
$x_0\in(0,\rho_{\min})$ such that $\Phi(x_0)=1$.

We next show that 
$$x_0=\radconv_{G\ast H,T\sqcup U}.$$
By Part \eqref{p:free-prod2}, for every $0 \le x < \radconv_{G\ast H,T\sqcup U}$, 
we have 
$$\Phi(x)=Q_{G*H, T \sqcup U}(x) <1.$$
Since $\Phi(x_0)=1$, it follows 
$$x_0 \ge \radconv_{G\ast H,T\sqcup U}.$$
Since $x_0 < \rho_{\min}$, we obtain
$$\radconv_{G\ast H,T\sqcup U} < \rho_{\min}.$$
This proves Part \eqref{p:free-prod1}.

It follows that
$$e(G*H,T \sqcup U) = \frac{1}{\radconv_{G\ast H,T\sqcup U}} > \frac{1}{\rho_{\min}}=\max\{e(G,T), e(H,U)\}.$$

Since $\radconv_{G\ast H,T\sqcup U} < \rho_{\min}$, 
the function $\Phi$ is continuous at $\radconv_{G\ast H,T\sqcup U}$. 
Hence, by Part \eqref{p:free-prod2} and Lemma \ref{lem:spherical-basic} \eqref{p:hyp}, 
$$\Phi(\radconv_{G\ast H,T\sqcup U}) =\lim_{x \nearrow \radconv_{G\ast H,T\sqcup U} } \Phi(x) = \lim_{x \nearrow \radconv_{G\ast H,T\sqcup U} } 
Q_{G\ast H,T\sqcup U }(x) =1.$$

By the uniqueness of $x_0$, we have $x_0=\radconv_{G\ast H,T\sqcup U}$.
The proof is complete. 
\end{proof}



The following finite-factor version is an immediate consequence of Proposition \ref{lem:spherical-free}.
\begin{cor}[The $Q$-formula for finite free products]
\label{cor:spherical-free-finite}
Let $k\ge 2$, and let $G_i$ be infinite groups with finite generating
sets $T_i$ for $1\le i\le k$. Set
$$
    G:=G_1*\cdots *G_k,
    \qquad
    T:=\bigsqcup_{i=1}^k T_i,
$$
and
$$
    \rho_{\min}:=\min_{1\le i\le k}\rho_{G_i,T_i}.
$$
Define
$$
    \Phi(x):=\sum_{i=1}^k Q_{G_i,T_i}(x),
    \qquad
    0\le x<\rho_{\min}.
$$
Then
$$
    \rho_{G,T}<\rho_{\min}.
$$
Moreover, for every $0\le x<\rho_{G,T}$, we have
$$
    Q_{G,T}(x)=\sum_{i=1}^k Q_{G_i,T_i}(x).
$$
There exists a unique $x_0\in(0,\rho_{\min})$ such that
$$
    \Phi(x_0)=1,
$$
and
$$
    x_0=\rho_{G,T}.
$$
\end{cor}

\begin{proof}
Apply Proposition~\ref{lem:spherical-free} inductively to
$$
    (G_1*\cdots *G_{k-1})*G_k.
$$
The $Q$-identity and the strict inequality on radii of convergence are preserved at
each step, and the characterization of $\rho_{G,T}$ as the unique solution of
$$
    \sum_{i=1}^k Q_{G_i,T_i}(x)=1
$$
follows immediately.
\end{proof}

The advantage of the $Q$-function approach is that
although $Q_{G,T}$ itself is defined only 
for $0 \le x< \rho_{G,T}$, the sum of the factor
$Q$-functions is defined on the larger 
interval 
$0 \le x < \min_{1\le i\le k} \rho_{G_i,T_i}$.

We will repeatedly use this enlarged interval
to compare $e(G,T)$ with a prescribed reference value. This principle plays a key role in the proofs of  Theorems~\ref{thm:main} and~\ref{thm:main.eq}.

\bigskip

We record the spherical growth series $P_{\mathbb{Z}, U}$ and the associated functions $Q_{\mathbb{Z}, U}$ of $\mathbb{Z}$ for various generating sets $U$ that are used repeatedly as building blocks.
The following formulae can be derived from direct computations using \cite{GAP4}.

\begin{exmp}[Growth series of $\mathbb{Z}$]\label{exmp:cyclic}
Note that for all power series below, we have $\rho_{\mathbb{Z},U}=1$,
where $U=\{1\}$, $\{1, 2\}$, $\{1, 3\}$, and $\{2,3\}$ in $\Z$. 

For $x \in [0,1)$, we have the following formulas. 

\begin{equation}\label{eq:1}
P_{\mathbb{Z},\{1\}}(x) = 1+\sum_{i\ge1}2x^i=\frac{1+x}{1-x},
\quad
Q_{\mathbb{Z},\{1\}}(x) = \frac{2x}{1+x},
\end{equation}
\begin{equation}\label{eq:12}
P_{\mathbb{Z},\{1,2\}}(x) = 1+\sum_{i\ge1}4x^i= \frac{1+3x}{1-x},\quad
Q_{\mathbb{Z},\{1,2\}}(x) = 
\frac{4x}{1+3x},
\end{equation}
\begin{equation}\label{eq:13}
P_{\mathbb{Z},\{1,3\}}(x) = 1+4x+\sum_{i\ge2}6x^i= \frac{1+3x+2x^2}{1-x}
,\quad
Q_{\mathbb{Z},\{1,3\}}(x) = 
\frac{4x+2x^2}{1+3x+2x^2},
\end{equation}
\begin{equation}\label{eq:23}
P_{\mathbb{Z},\{2,3\}}(x) = 1+4x+8x^2+\sum_{i\ge3}6x^i= \frac{1+3x+4x^2-2x^3}{1-x},
\end{equation}
\[
Q_{\mathbb{Z}, \{2, 3\}}(x)=
\frac{4x+4x^2-2x^3}{1+3x+4x^2-2x^3}.
\]
\end{exmp}
We denote the standard free basis of $F_d$ by
\[
A_d:=\{a_1,\ldots,a_d\}.
\]

The following two examples will be used in our later estimates. 
\begin{exmp}\label{exmp:free}
From 
Corollary~\ref{cor:spherical-free-finite} and Example~\ref{exmp:cyclic} \eqref{eq:1}, we have that
\begin{equation}\label{eq:d}
Q_{F_d,A_d}(x)=d\cdot Q_{\mathbb{Z},\{1\}}(x)=\frac{2dx}{1+x}\quad
\text{ for each }x\in\left[0,\frac1{2d-1}\right).
\end{equation}

For the generating set $\Delta_2:=\{a_1, a_2, a_1a_2\}$ of $F_2$, 
a direct computation yields
$$|S_{F_2,\Delta_2}(n)|=6\cdot 4^{n-1}$$
for $n\ge 1$.
This shows that
$\rho_{F_2, \Delta_2}=1/4$ and 
\begin{equation}\label{eq:triangle}
P_{F_2, \Delta_2}(x)=\frac{1+2x}{1-4x},\quad
Q_{F_2, \Delta_2}(x)=\frac{6x}{1+2x}
\end{equation}
for all $x \in [0, 1/4)$.
\end{exmp}

\bigskip

For two real(-coefficient) power series $A=\sum_{i=0}^\infty a_iX^i$ and $B=\sum_{i=0}^\infty b_iX^i$, we write 
\[
A\preceq B
\]
if $a_i\le b_i$ for all $i \ge 0$. We write $A\prec B$ if $A\preceq B$ and $A\ne B$.

We will later need the following lemma to compare growth rates.
Let $\mathbb{Z}_+$ denote the set of positive integers.
\begin{lem}\label{lem:sph-cyc}
Let $U\subseteq\mathbb{Z}_+$ be a finite generating set of the additive group $\mathbb{Z}$. Then $\rho_{\mathbb{Z},U}=1$. Moreover, 
\begin{enumerate}
\item \label{p:E1}
$P_{\mathbb{Z}, U} \succeq P_{\mathbb{Z}, \{1\}}$ with 
equality if and only if $U=\{1\}$. 
In particular, 
$Q_{\mathbb{Z}, U}(x) \ge Q_{\mathbb{Z}, \{1\}}(x)$ for every $x \in (0, 1)$. Also,  $Q_{\mathbb{Z}, U}(x) = Q_{\mathbb{Z}, \{1\}}(x)$ for some $x\in (0, 1)$ if and only 
if $U=\{1\}$.
\item\label{p:E2-1} 
If $|U|\ge 2$ and $U\ne\{1,2\}$, then 
$P_{\mathbb{Z},U} \succ P_{\mathbb{Z}, \{1, 2\}}.$
Hence, $Q_{\mathbb{Z}, U}(x) > Q_{\mathbb{Z}, \{1,2\}}(x)$ for every $x\in (0,1)$.
\item\label{p:E2-2} 
If $U$ is none of $\{1\}$, $\{1, 2\}$, or $\{1, 3\}$, 
then 
$
P_{\mathbb{Z},U}\succ
P_{\mathbb{Z},\{1,3\}}
$.
Hence, $Q_{\mathbb{Z},U}(x)>
Q_{\mathbb{Z},\{1,3\}}(x)$
for every $x\in (0,1)$.
\end{enumerate}
\end{lem}
\begin{proof}
We have already observed that $\rho_{\mathbb{Z},U}=1$ in Example \ref{exmp:cyclic}.
Since the function $t \mapsto 1-1/t$ is strictly increasing 
on $(0,\infty)$, each inequality for $Q$ immediately follows from the corresponding inequality 
for $P$. Thus it suffices to prove the claims
for  $P$.
Let $M:=\max U$.

We show \eqref{p:E1}. For every $n \ge 1$, the two elements $\pm nM$ have $U$-length exactly $n$. Hence, 
$$P_{\mathbb Z, U} \succeq P_{\mathbb Z,\{1\}}.$$
If equality holds, comparing the coefficients yields $|U|=1$, hence $U=\{1\}$.
Therefore, if $U\not=\{1\}$, then $P_{\mathbb Z, U} \succ P_{\mathbb Z,\{1\}}.$
Hence, $Q_{\mathbb{Z}, U}(x) = Q_{\mathbb{Z}, \{1\}}(x)$ for some $x\in (0, 1)$ if and only 
if $U=\{1\}$.

We show \eqref{p:E2-1}.
Take $0<a<M$ such that $\{a, M\} \subset U$.
Observe that $\pm Mn$, $\pm(M(n-1)+a)$ are four distinct elements of length $n$ for every $n \ge 1$.
By Example \ref{exmp:cyclic} \eqref{eq:12},
this shows the required inequality
\[
P_{\mathbb{Z},U}\succeq
1+\sum_{i\ge 1}4X^i=
P_{\mathbb{Z}, \{1, 2\}}.
\]
We show the strict inequality
$P_{\mathbb{Z},U} \succ P_{\mathbb{Z}, \{1, 2\}}$.
If equality holds,
then the number of elements of length $n$ is $4$ for every $n\ge 1$.
Thus, $U=\{a, M\}$, where $a$ and $M$ are relatively prime to generate $\mathbb{Z}$.
Moreover, the element $M-a$ should not belong to the radius-$2$ sphere;
this forces $M-a=a$, from which $M=2a$ and thus $a=1$.
Therefore, $U=\{1, 2\}$.

We show \eqref{p:E2-2}.
First, consider the case that $|U|\ge 3$.
Take $0<a<b<M$ in $U$.
Then, the number of elements of length $1$ is at least $6$.
For every $n \ge 2$, the elements $Mn$, $M(n-1)+a$, and $M(n-1)+b$, are distinct and have length $n$. 
Indeed, they are strictly greater than $M(n-1)$, having length at least $n$.
Taking into account their inverses, we have at least $6$ elements of length $n$.
Hence, by Example \ref{exmp:cyclic} \eqref{eq:12} and \eqref{eq:13}, we obtain 
\begin{equation*}
P_{\mathbb{Z}, U}\succeq 1+\sum_{i\ge 1}6X^i\succ P_{\mathbb{Z},\{1,3\}}
\succ P_{\mathbb{Z},\{1,2\}}.
\end{equation*}

Next, assume that $|U|=2$.
We have $U=\{a, M\}$ where $0<a<M$.
Then $a$ and $M$ are relatively prime since they generate $\mathbb{Z}$.
By assumption, $U\neq \{1,2\}$.
Thus, we may assume that $M\neq 2a$.
We will show that for every $n\ge 2$, the number of elements of length $n$ is at least $6$.
To show this, we will find $3$ positive elements of length $n$.

Suppose that $M<2a$. 
Then, the $3$ elements $Mn$, $M(n-1)+a$, and $M(n-2)+2a$ have length $n$ since they are strictly greater than $M(n-1)$.
Furthermore, noting that $0<M-a<a$, we have at least $4$ positive elements of length $2$: 
$$M-a < 2a <  M+a < 2M.$$
Thus, by Example \ref{exmp:cyclic} \eqref{eq:13} and \eqref{eq:23}, we have 
$$P_{\mathbb{Z},U}\succeq 1+4X + 8 X^2 + \sum_{i\ge3} 6X^i = P_{\mathbb{Z},\{2,3\}} \succ  P_{\mathbb{Z},\{1,3\}}.$$

Finally, we assume that $M>2a$. 
The elements $Mn$ and $M(n-1)+a$ have length $n$.
We set $u_n:=M(n-1)-a$. 

\begin{claim*}
We have $|u_n|_U=n$.
\end{claim*}

By definition, $|u_n|_U\le n$.
We show that $|u_n|_U \ge n$.
Argue by contradiction; assume that $u_n=iM+ja$ with $|i|+|j|\le n-1$.
Since $(n-1-i)M=(j+1)a$, and $a$ and $M$ are relatively prime, there exists an integer $k$ such that 
\[
j+1=kM \quad \text{and} \quad n-1-i=ka.
\]
We have $k\ne0$ for otherwise, we would have $|i|+|j|=n-1+1=n$.
If $k<0$, then $i>n-1$, whence we have the contradiction $|i|+|j|>n-1$.
If $k>0$ and $n-1\ge ka$, then we have $j=kM-1\ge 0$ and $i=n-1-ka\ge 0$; in this case, we have 
$$0\le i = n-1-ka\le n-2,$$
and 
$$|i|+|j|=i+j = n-2+k(M-a)\ge n,$$
where we have used $M\ge 2a+1$.
It remains to consider the case when $k>0$ and $n-1<ka$.
Then, we have
\[
|i|+|j|=-i+j = k(M+a)-n>2n-n=n.
\]
Thus, $u_n$ must have the length strictly greater than $n-1$ for every $n\ge 2$, as required.
This completes the proof of the claim.

By the claim, we have at least $6$ elements of length $n$. 
Note that $a<M-a$, and that $2a= M-a$ if and only if $(a, M)= (1, 3)$.
Since we are assuming $U\neq \{1, 3\}$, we have at least $4$ positive elements of length $2$: $2a$, $M-a$, $M+a$, and $2M$. 
It follows that by Example \ref{exmp:cyclic} \eqref{eq:13} and \eqref{eq:23},
$$P_{\mathbb{Z},U} \succeq
1 + 4 X + 8 X^2 + 6\sum_{i\ge3} X^i 
=P_{\mathbb{Z},\{2,3\}} \succ P_{\mathbb{Z},\{1,3\}},
$$
showing the strict inequality for all $x\in (0, 1)$.
\end{proof}

We show Proposition \ref{prop:cyclic} stated in Section \ref{sec:outline}.

\begin{proof}[Proof of Proposition~\ref{prop:cyclic}] 
First, \eqref{p:12} and \eqref{p:13} follow from direct computations.

Second, we show \eqref{p:pq}. 
For each relatively prime  $1\le p<q$, let
$$
\Phi^d_{p,q}(x) := Q_{\mathbb{Z},\{p,q\}}(x)+
Q_{F_{d-1},A_{d-1}}(x), \qquad
0\le x<\frac1{2d-3}.$$
Using Example \ref{exmp:cyclic} \eqref{eq:12}, \eqref{eq:13}, and \eqref{eq:d},
we see for each $x\in[0,1/(2d-3))$ that
\begin{align*}
\Phi^d_{1,2}(x) &= \frac{4x}{1+3x}+\frac{2(d-1)x}{1+x},\\
\Phi^d_{1,3}(x) &= \frac{4x+2x^2}{1+3x+2x^2}+\frac{2(d-1)x}{1+x}.
\end{align*}
We have a unique positive real solution $\rho^d_{p,q}$
of $\Phi^d_{p,q}(x)=1$; indeed, by solving the corresponding quadratic equations, we have
\begin{align*}
\rho_{1,2}^d&=\frac{\sqrt{d^2+4 d-4}-d+1}{6 d-5}=\frac1{\alpha_d(1,2)},\\
\rho_{1,3}^d&=\frac{\sqrt{4 d^2+12 d-15}-2 d+1}{8 (d-1)}=\frac1{\alpha_d(1,3)}.
\end{align*}
Lemma \ref{lem:sph-cyc} \eqref{p:E2-2} further implies that
$\Phi^d_{p,q}(x)\ge\Phi^d_{1,3}(x)$ for $\{p,q\}\ne\{1,2\}$,
where equality holds for some $x>0$ if and only if $\{p,q\}=\{1,3\}$.
From the same lemma, we further see that
 $$\rho_d(p,q)<\rho_d(1,3)<\rho_d(1,2)$$ when $\{p,q\}\ne\{1,2\},\{1,3\}$.
Therefore, Lemma \ref{lem:spherical-basic} \eqref{p:free-b-prod1} yields \eqref{p:pq}.

Finally, we show \eqref{p:b}.
Consider the generating set
$$\Delta_d:=\{a_1,a_2,a_1a_2, a_3,\ldots,a_d\}\subseteq F_d.$$
If $d=2$, then Example \ref{exmp:free} \eqref{eq:triangle} gives
$$\rho_{F_2,\Delta_2}=\frac{1}{4}=\frac{1}{\beta_2}=\frac{1}{\alpha_2(1,3)}.$$
Assume $d\ge 3$. 
Proposition \ref{lem:spherical-free} \eqref{p:free-prod2}
shows that $Q_{F_d, \Delta_d}(x)$ 
for $x \in [0, \rho_{F_d, \Delta_d})$ equals
\begin{align*}
\Phi(x)
&:= Q_{F_2,\Delta_2}(x)+ (d-2)Q_{\mathbb{Z},\{1\}}(x)\\
&=\frac{6x}{1+2x} + \frac{2(d-2)x}{1+x}
=\frac{4x+2x^2}{1+3x+2x^2}+\frac{2(d-1)x}{1+x}.
\end{align*}
Since the equation $\Phi(x)=1$ has a unique positive solution 
\[
x_0=\frac{1}{\alpha_d(1,3)}
\]
in $(0, \rho_{F_2, \Delta_2})$,
by Proposition~\ref{lem:spherical-free} \eqref{p:phi-strict}, we conclude \eqref{p:b}.
\end{proof}

\section{Growth estimate for Magnus-marked groups}\label{sec:just-non-free}
In this section, we give a growth estimate for
Magnus-marked groups. 
This estimate will
be used to prove Proposition~\ref{prop:noncyc}.
The key is the following purely graph-theoretic proposition that guarantees that a regular graph has an expansion property when it is evenly covered by a family of spanning forests.
 We will apply this proposition to the Cayley graphs of Magnus-marked groups.

For a graph $\Gamma$, we denote by $V\Gamma$ the vertex set and by $E\Gamma$ the edge set. 
For every set $A \subseteq V\Gamma$, we define
\[
\partial_V A:=\{v \in V\Gamma \setminus A \mid \{v, w\} \in E\Gamma \text{ for some $w \in A$}\}.
\]
 Furthermore, a {\it spanning forest} $\mathcal{F}$ of $\Gamma$ is a subgraph of $\Gamma$ whose vertex set is $V\Gamma$ such that each connected component is a tree.
 For a spanning forest $\mathcal{F}$ of $\Gamma$,
 and for every $A \subseteq V\Gamma$,
 we define
 \[
 c_{\Gamma, \mathcal{F}}(A):=|\{\mathcal{T} \mid \text{$\mathcal{T}$ is a connected component of $\mathcal{F}$ such that $\mathcal{T}\cap A\neq \emptyset$}\}|.
 \]

\begin{prop}\label{prop:graph}
Fix integers $D\ge 3$ and $t\ge 2$. 
Let $\Gamma$ be a $D$-regular graph, and $\Gamma_0, \Gamma_1, \dots, \Gamma_t$ be $(D-2)$-regular spanning forests of $\Gamma$ such that each $e\in E\Gamma$ is included exactly in $t$ different forests $\Gamma_j$.

Then,
for every finite set $A \subset V\Gamma$,
we have
\begin{equation}
|\partial_V A|
\ge \left(1+\frac1t\right)(D-4)|A|
+\frac{2}{t}\sum_{j=0}^t c_{\Gamma, \Gamma_j}(A),
\end{equation}
\end{prop}

\begin{proof}
Let us define an edge labeling of $\Gamma$:
each $e\in E\Gamma$ is labeled by $i\in \{0, 1, \dots ,t\}$ if $e$ is not included in $E\Gamma_i$.
By assumption, each edge of $\Gamma$ has a unique label.

For every finite set $A\subseteq V\Gamma$,  we denote by $\Gamma[A]$ the induced subgraph of $\Gamma$ on the vertex set $A$.
We fix $A \subseteq V\Gamma$ and introduce the following notation:
\begin{align*}
\partial_E A&:=\{\{v, w\} \in E\Gamma \mid v\in V\Gamma\setminus A, w \in A\},\\
\partial_V^j A &:=\{v \in V\Gamma \setminus A \mid \text{$\{v, w\} \in E\Gamma_j$ for some $w \in A$}\},\\
\partial_E^j A &:=\{\{v, w\} \in E\Gamma_j \mid v\in V\Gamma \setminus A, w\in A\}.\end{align*}
Furthermore, 
let $\hat\Gamma_j(A)$ be the subgraph of $\Gamma_j$ obtained by attaching the edges in $\partial_E^j A$ to $\Gamma_j[A]$, 
together with their endpoints outside $A$.

We denote by $c_j(A)$ and by $\hat c_j(A)$ the number of connected components in $\Gamma_j[A]$ and in $\hat\Gamma_j(A)$, respectively.

By assumption, the subgraph $\Gamma_j$ of $\Gamma$ is a $(D-2)$-regular forest.
Hence $\Gamma_j[A]$ is a finite forest.
We observe that
\begin{equation}\label{eq:2t}
\frac{(D-2)|A|-|\partial_E^j A|}2=|E\Gamma_j[A]|
=|A|-c_j(A).
\end{equation}
We obtain the first equality by counting the $\Gamma_j$-edges incident to vertices of $A$ (the total count is $(D-2)|A|$), and using that each internal edge is counted twice and each boundary edge once. 
The second equality follows since $\Gamma_j[A]$
is a finite forest with $c_j(A)$ components.

Furthermore, since $\hat\Gamma_j(A)$ is a forest, 
$$
    \hat c_j(A)= |V\hat\Gamma_j(A)|-|E\hat\Gamma_j(A)|
=(|A|+|\partial_V^jA|)-(|E\Gamma_j[A]|+|\partial_E^jA|),$$
and hence,
\begin{equation}\label{eq:cq}
c_j(A) - \hat c_j(A) = |\partial_E^j A|-|\partial_V^j A|.
\end{equation}
Eliminating $|\partial_E^jA|$ from~\eqref{eq:2t} and the above equality, we have
\begin{equation}\label{eq:pv}
 |\partial_V^j A|=(D-4)|A|+c_j(A)+\hat c_j(A).\end{equation}

Let us define a partition 
$\partial_VA=\partial_V^IA\sqcup\partial_V^{II}A$,
where 
\begin{align*}
\partial_V^I A&:=\{v \in \partial_VA \mid \text{all edges joining $v$ to $A$ have the same label}\},\\
\partial_V^{II}A&:=\{v \in \partial_V A\mid \text{not all edges joining $v$ to $A$ have the same labels}\}.
\end{align*}
By the definition of edge labeling, each $v\in \partial_V^I A$ 
belongs to exactly $t$ of 
$$\partial_V^0 A,
\partial_V^1 A,\ldots,
\partial_V^t A,$$
while $v\in \partial_V^{II}A$ belongs to all of the above sets.
Hence, we have
\begin{equation*}
\sum_{j=0}^t|\partial_V^jA|
=t|\partial_V^IA|+(t+1)|\partial_V^{II}A|
=t|\partial_V A|+|\partial_V^{II}A|.
\end{equation*}
Summing ~\eqref{eq:pv} over $j=0, \dots,t$ and using the above equality, we obtain 
\begin{equation}\label{eq:td}
t|\partial_VA|+|\partial_V^{II}A|
= (t+1)(D-4)|A|+C(A)+\hat C(A),
\end{equation}
where
\[
C(A):=\sum_{j=0}^t c_j(A) \quad \text{and} \quad \hat C(A):=\sum_{j=0}^t \hat c_j(A).
\]

Let $d_j(v):=|\{\{v, w\} \in E\Gamma_j \mid w \in A\}|$.
By \eqref{eq:cq},
we have
$$
c_j(A)-\hat c_j(A)=
|\partial_E^j A|-
|\partial_V^j A|
=\sum_{v\in \partial_V^j A} (d_j(v)-1)
=\sum_{v\in \partial_V A} \max(d_j(v)-1,0).
$$

Suppose that $v\in \partial_V^{II}A$. 
Then, there exist at least two edges joining $v$ to $A$ with different labels.
Thus, for at least one $k \in \{0, 1, \dots, t\}$ where $t\ge 2$,
we have $d_k(v) \ge 2$.


Summing the above equality over $j=0, \dots, t$, 
combined with the above observation on $v\in \partial_V^{II}A$, we obtain
$$
C(A)-\hat C(A)=
\sum_{v\in \partial_V A} \sum_{j=0}^t \max(d_j(v)-1,0)\ge|\partial_V^{II}A|.$$
Using this inequality in ~\eqref{eq:td}, we obtain that
\begin{equation}\label{eq:tdv}
t|\partial_V A|\ge
(t+1)(D-4)|A|+2\hat C(A).
\end{equation}

Note that if two elements of $A$ belong to distinct connected components of $\Gamma_j$, 
then they belong to distinct components of $\widehat\Gamma_j(A)$,
whence $\hat c_j(A)\ge c_{\Gamma, \Gamma_j}(A)$.
It follows that
\[
\hat C(A)=\sum_{j=0}^t \hat c_j(A) \ge \sum_{j=0}^tc_{\Gamma, \Gamma_j}(A).
\]
Therefore, by \eqref{eq:tdv},
we obtain 
\begin{equation*}
|\partial_V A|
\ge \left(1+\frac1t\right)(D-4)|A|
+\frac{2}{t}\sum_{j=0}^t c_{\Gamma, \Gamma_j}(A),
\end{equation*}
as desired.
\end{proof}

We now apply the above proposition to the Cayley graphs of Magnus-marked groups.

For a countable group $G$ and a subgroup $H\le G$, we define the \emph{coset counting function} $c_{G/H}$ as the map from the power set $2^G$ to $\mathbb{Z}_{\ge 0}\cup\{\infty\}$ such that 
$$ c_{G/H}(A):=|\{gH\in G/H \;|\; gH\cap A\ne\varnothing\}|.$$

\begin{prop}[Expansion estimate for Magnus markings, cf. Lemma~\ref{lem:nonsplit_intro}]\label{lem:nonsplit}
Let $T=\{y_0,\ldots,y_t\}$ be a Magnus marking of a group $G$ such that $t=|T|-1\ge2$ and $T\cap T^{-1}=\varnothing$. 
For $i=0,1,\ldots,t$, set
$$G_i:=\form{T\setminus\{y_i\}}.$$ 
\begin{enumerate}
    \item\label{p:expansion}
For every finite subset $A$ of $G$,  we have that
$$
|A(T\cup T^{-1})\setminus A|
\ge \left(2t-\frac{2}{t}\right)|A|+\frac{2}{t}\sum_{i=0}^t c_{G/G_i}(A).
$$
\item\label{p:Magnus_growth}
Set $$\theta_t:=1+2t-\frac{2}{t} \ge4 .$$ 
Then we have that
\begin{align*}
    e(G,T)\ge \theta_t, \quad
    \rho_{G,T} \le \frac{1}{\theta_t}, \quad \text{and} \quad
    P_{G,T}\succeq 1+2(t+1)\sum_{n\ge1} \theta_t^{n-1} X^n.
\end{align*}

Moreover, for each $x\in [0,\rho_{G,T})$,
we have
\begin{equation}\label{eq:qgt}
Q_{G,T}(x)\ge \frac{2(t+1)x}{1+(1+2/t)x}.
\end{equation}
Equality holds in \eqref{eq:qgt} for some $x\in (0,\rho_{G,T})$
if and only if
$$|S_{G,T}(n)|= 2(t+1)\theta_t^{n-1}$$
for all $n\ge1$.
In this case, 
$G_i=G$ for every $i=0, \dots, t$.
\end{enumerate}
\end{prop}

\begin{proof}
First, we show \eqref{p:expansion}.
We define $\Gamma:=\cay(G,T)$.
Note that $\Gamma$ is a $(2t+2)$-regular graph since $T\cap T^{-1}=\varnothing$ and $|T|=t+1$.
For each $j=0, 1, \dots, t$,
let $\Gamma_j$ denote the subgraph of $\Gamma$ with the same vertex set, containing only the edges defined by $$T_j:=T\setminus\{y_j\}.$$ 
Since $T$ is a Magnus marking of $G$,
the set $T_j$ freely generates a free subgroup of $G$.
Thus, $\Gamma_j$ is a spanning forest of $\Gamma$.
Furthermore, each edge in $\Gamma$ is included exactly in $t$ different $\Gamma_j$.
We also note that $c_{G/G_j}(A)=c_{\Gamma, \Gamma_j}(A)$ for every $A\subseteq G$.
Therefore, Proposition \ref{prop:graph} implies \eqref{p:expansion}.

Second, we show \eqref{p:Magnus_growth}.
Now we apply part \eqref{p:expansion} to the balls $B_T(n)$.
Recall that
$$\theta_t=1+2t-\frac{2}{t}.$$
Let $$c_t:=2\left(\frac{t+1}{t}\right), \, \,  b_n:=|B_T(n)|, \,\,s_n:=|S_T(n)|=|\partial_VB_T(n-1)|.$$
We have that $b_0=1$. Applying the above inequality to $A=B_T(n)$, we have
$$s_{n+1} \ge 2\left(t-\frac{1}{t}\right)b_n + 2\left(\frac{t+1}{t}\right)$$
This implies that
$$b_{n+1}=b_n+s_{n+1}\ge \theta_t b_n+c_t.$$
Applying this inequality repeatedly, we obtain  
\[
b_n\ge \theta_t^n + \frac{\theta_t^n-1}{\theta_t-1} c_t,
\]
whence,
\[
s_n\ge (\theta_t-1)b_{n-1}+c_t
\ge (\theta_t-1)\theta_t^{n-1}+\theta_t^{n-1} c_t=2(t+1)\theta_t^{n-1}.
\]

Since $b_n \ge \theta_t^n$, we have
$$e(G,T) \ge \theta_t \, \, \text{and} \, \, \rho_{G,T} \le \frac{1}{\theta_t}.$$
Also, the spherical growth series for $(G, T)$ satisfies
\begin{equation}\label{eq:pgt}
P_{G, T}\succeq 1+2(t+1)\sum_{n\ge1}\theta_t^{n-1}X^n
=\frac{1+(1+2/t)X}{1-(1+2t-2/t)X}.\end{equation}
For each $x \in [0,\rho_{G,T}) $, we have
$$
{Q_{G,T}(x)}\ge 
 1 - \frac{1-(1+2t-2/t)x}{1+(1+2/t)x}=\frac{2(t+1)x}{1+(1+2/t)x}.
$$

Finally, suppose that equality holds in this inequality 
for some $x\in(0,\rho_{G,T})$. 
Then, equality holds in (\ref{eq:pgt}) at that $x$,
and
we have coefficient-wise equality in \eqref{eq:pgt}.
Hence, $|S_{T}(n)|= 2(t+1)\theta_t^{n-1}$
for all $n>0$. 
The converse is immediate. 

It remains to prove the final assertion. 
Applying Part (\ref{p:expansion})
to $A=B_T(n)$, we have 
$$|S_T(n+1)| \ge (\theta_t-1)|B_T(n)| + \frac{2}{t} \sum_{i=0}^t c_{G/G_i}(B_T(n)).$$
Also, we have $|S_T(0)|=1$ and  $|S_T(i)|=2(t+1)\theta_t^{i-1}$ for $i \ge 1$. 
Therefore,
$$|B_T(n)|=\sum_{i=0}^n|S_T(i)|=1+2(t+1)\frac{\theta_t^n-1}{\theta_t-1}=
\frac{t\theta_t^n-1}{t-1},$$ 
since $\theta_t-1=2(t-1)(t+1)/t.$

Substituting these expressions into the above inequality, we obtain 
$$2(t+1) \theta_t^n \ge 2(t+1) \theta_t^n - 2\left(\frac{t+1}{t}\right) + \frac{2}{t} \sum _{i=0}^t c_{G/G_i}(B_T(n))$$
This implies
$$\sum_{i=0}^t c_{G/G_i}(B_T(n))\le t+1.$$

Now, since $B_T(n)$ is non-empty, we have $c_{G/G_i}(B_T(n)) \ge 1$
for every $i=0, \dots, t$. Therefore, the above inequality implies that 
$c_{G/G_i}(B_T(n))=1$ for every $i$ and every $n \ge0$.
Taking $n=1$, the ball 
$B_T(1)$ meets exactly one coset of $G_i$. Since $1 \in B_T(1)$, this coset is $G_i$ itself. Hence $B_T(1)$
is contained in $G_i$. Since $T \subset B_T(1)$, 
we have $T \subset G_i$. This implies $G_i=G$ for every $i=0, \dots, t$.
\end{proof}

We now prove Proposition \ref{prop:noncyc} stated in Section \ref{sec:outline}.
We denote by $\rk G$ the \emph{rank} of $G$, the smallest number of generators for $G$.
\begin{proof}[Proof of Proposition~\ref{prop:noncyc}]
Let us pick an inclusion-minimal $W\subseteq U$ such that $$1\le t:=\rk\form{W}<|W|.$$
Since 
$$d = \rk\form{U} \le \rk \form{W}+\rk\form{U\setminus W}\le t+|U\setminus W|\le t+(d+1)-|W|\le d,$$
we have that $|W|=t+1$ and $\rk \form{U\setminus W}= d-t$. Hence 
$$F_d=\form{U}=\form{W}\ast \form{U\setminus W}\cong \form{W}\ast F_{d-t}.$$
Here, the natural epimorphism
$
\form{ W} * \form{ U\setminus W} \longrightarrow \form{ U } 
$
is an isomorphism, since both groups are free of rank $d$ and $F_d$ is Hopfian.

Since $U$ contains no commuting pair, we have $t\ge2$. Indeed, if $t=1$, then 
$|W|=2$ and the two elements of $W$
belong to the cyclic group $\form{ W }$, hence commute, a contradiction.
By minimality, $W$ is a Magnus marking for $H:=\form{W}$.

Recall
$$\theta_t=1+2t-\frac{2}{t}.$$
Applying Proposition \ref{lem:nonsplit} to the Magnus marking $W$ of $H$, we obtain  
$$e(H,W) \ge \theta_t, \, \, 
\rho_{H,W}\le \frac{1}{\theta_t},$$ and 
\begin{equation}\label{eq:qhw}
Q_{H,W}(x)\ge  \frac{2(t+1)x}{1+(1+2/t)x}
\end{equation}
for all $x\in [0,\rho_{H,W})$.

We first consider the case $d=t$.
Then $W=U$ and $H=F_d$, hence
$$e(F_d,U)=e(H,W)\ge \theta_d.$$
We next show
$$\theta_d \ge \beta_d.$$
Indeed, by Proposition \ref{prop:cyclic} (\ref{p:13}) and (\ref{p:b}), $\beta_d$ is the unique positive root of
$$p_d(y)=y^2-(2d-1)y-(4d-4).$$
A direct computation gives
$$p_d(\theta_d)=\frac{2(d-1)(d-2)}{d^2} \ge 0.$$
Since  $\theta_d>0$,
it follows that  $\theta_d \ge \beta_d$,
with equality  if and 
only if $d=2$.

Combining the two inequalities, we have $e(F_d,U) \ge \beta_d$.
This completes the case $d=t$.

We assume $d>t$ from now on. 
For the free product decomposition $F_d=H*F_{d-t}$, set
\begin{align*}
\rho_{\min}&:=\min(\rho_{H,W},\rho_{F_{d-t},A_{d-t}}),\\
\Phi(x)&:=Q_{H,W}(x)+Q_{F_{d-t},A_{d-t}}(x), \, \, x \in [0,\rho_{\min} ).
\end{align*}

Since $H$ and $F_{d-t}$ are infinite,  Proposition~\ref{lem:spherical-free}  ~\eqref{p:phi-strict} implies that 
$$x_1:=\rho_{F_d,U}$$ is the 
unique solution in 
$(0,\rho_{\min})$ of 
$\Phi(x)=1.$
Moreover, $$e(F_d,U)=\frac{1}{x_1}.$$

Combining equation~\eqref{eq:qhw} with Example \ref{exmp:free} \eqref{eq:d}, we have 
\begin{equation}\label{eq:17}
\Phi(x)
\ge \frac{2(t+1)x}{1+(1+2/t)x}+\frac{2(d-t)x}{1+x}, \, \, x\in [0,\rho_{\min}).
\end{equation}
Define $\Phi_0$ by putting $t=2$ in the preceding  lower bound:
$$\Phi_0(x)=\frac{6x}{1+2x}+\frac{2(d-2)x}{1+x}, \, \, x \in [0,\infty).$$

We claim
\begin{equation}\label{eq:phi-phi0}
\Phi(x)\ge\Phi_0(x)
\end{equation}
for all $x\in [0,\rho_{\min})$.
Indeed, 
$$
\frac{2(t+1)x}{1+(1+2/t)x}+\frac{2(d-t)x}{1+x} \ge \frac{6x}{1+2x}+\frac{2(d-2)x}{1+x}
$$
for all $x\in[0,1)$.
A direct computation shows that 
\begin{equation}\label{eq:LHSRHS}
(\mathrm{LHS})-(\mathrm{RHS})
=\frac{2 (t-2) (1-x) x^2}{(x+1) (2 x+1) ((t+2) x+t)}.
\end{equation}
This quantity is non-negative for $x\in[0,1)$ since $t\ge 2$.

Since $\rho_{\min} \le 1$, combining
this with \eqref{eq:17} proves inequality \eqref{eq:phi-phi0}. 

\bigskip
We have the free product decomposition $F_d=F_2 * F_{d-2}$ with
$\Delta_d=\Delta_2 \sqcup A_{d-2}$. 
By Proposition \ref{lem:spherical-free},
$$\Phi_0(x)=Q_{F_2,\Delta_2}(x)+Q_{F_{d-2},A_{d-2}}(x), \, \, x \in [0,\rho_0),$$
where
$$\rho_0=\min\Bigg(\rho_{F_2, \Delta_2}, \rho_{F_{d-2},A_{d-2}}\Bigg)=\min\Bigg(\frac{1}{4}, \frac{1}{2d-5}\Bigg).$$

Let
$$x_0:=\rho_{F_d,\Delta_d}=\frac{1}{\beta_d}$$
be the unique solution in $(0,\rho_0)$
of $\Phi_0(x)=1$.
Since $x_1<\rho_{\min}$, by \eqref{eq:phi-phi0},
\[
1=\Phi(x_1) \ge \Phi_0(x_1).
\]
Therefore, $x_1 \le x_0$ since $\Phi_0$ is increasing on $[0, \infty)$. 
Hence,
we have
$$e(F_d, U) = \frac{1}{x_1}\ge \frac{1}{x_0} =\beta_d.$$
This completes the proof. 
\end{proof}

We record the following consequence of the proof
for later use. 
\begin{rem}\label{rem:concyc-eq}
    Suppose that equality holds in 
    Proposition \ref{prop:noncyc}, and 
    let $W \subseteq U$ be the inclusion-minimal subset
    chosen in the proof. Put $H=\form{ W }$,
    and $t=\rk H =|W|-1$.
    Recall that 
    $F_d=H*\form{ U \backslash W }$,
    and that $U \backslash W$ is a free basis of 
    the second factor.
Then 
\begin{enumerate}
    \item $t=2$, $ |W|=3$, $e(H,W)=4.$
\item If $d>t$, then  $\form{ W \backslash \{w\}} =H$ for every $w \in W$.
\end{enumerate}

    Indeed, suppose first that $d=t$. Then equality in the two inequalities
    used in the proof implies that $\theta_t=\beta_t$, and hence $d=2$,
    that is, $t=2$. Therefore,
    $$e(H,W)=e(F_d,U)=\beta_2=4.$$

    Suppose $d>t$. 
    Equality implies that $x_1=x_0$, and therefore 
    $\Phi(x_0)=\Phi_0(x_0)=1$. Then \eqref{eq:LHSRHS} forces $t=2$. Moreover, equality holds in the $Q$-estimate of Proposition \ref{lem:nonsplit} at the interior point $x_0$. By the equality statement therein, we have
    $$e(H,W)=\theta_2=4.$$
    The final assertion of Proposition~\ref{lem:nonsplit} gives $\form{ W \backslash \{w\}} =H$ for every $w\in W$. 
\end{rem}

\section{Reduction to cyclic factors and Proof of Theorem \ref{thm:main}}\label{sec:gaps}

For a group $G$, we say a subset $T\subseteq G$ is \emph{essentially reduced} 
if $T\cap T^{-1}=\varnothing$. Note that an essentially reduced $T$ does not contain $1$.
Every finite generating set of \(F_d\) is essentially equivalent to an essentially reduced generating set.

We now prove Theorem~\ref{thm:main}. Proposition~\ref{prop:noncyc}
controls the cyclically non-splitting case. The following lemma reduces
the remaining case when the growth rate is strictly below $\beta_d$ to a free product of cyclic factors.

\begin{lem}\label{lem:splitting}
Fix $d\ge 2$. 
Let $T$ be an essentially reduced finite generating set of $F_d$ with $|T|>d$.
Let $U$ be a subset of $T$ such that its image in $H_1(F_d;\bZ)$ is a free basis, and such that $U$ is a free basis of $\form{U}$.

If $e(F_d, T)<\beta_d$,
then 
$T$ is partitioned into disjoint sets $\{T_u\}_{u \in U}$ such that $\form{T}$ is isomorphic to the free product $\ast_{u \in U}\form{ T_u}$ and that each $T_u$ generates a cyclic group.
\end{lem}

\begin{proof}
First, we claim that every $s \in T \backslash U$ commutes with one of the elements in $U$.
To show the claim, set $V:=U\cup\{s\}$. Then $\rk \form{V}=d, d+1$ since
$$ d+1=|V|\ge \rk \form{V} \ge \rk H_1(F_d;\mathbb{Q})= d.$$ 
If $\rk \form{V}=d+1$, then $V$ freely generates $\form{V}$. 
Hence
\[
e(F_d,T)\ge e(\form{V},V)=2d+1>\beta_d,
\]
contradicting the assumption, whence $\rk\form{V}=d$.

Furthermore, assume that $s$ does not commute with any element in $U$. Then 
$V$ contains no commuting pairs. 
Applying Proposition~\ref{prop:noncyc} to $(\form{V},V)$,
we obtain  
$$e(F_d,T)\ge e(\form{V},V)\ge \beta_d=\alpha_d(1,3),$$
again contradicting the assumption. 
The claim is proved.

Next, from the claim, we see that
$$
T \subseteq \bigcup_{u \in U}\form{ p_u},
$$
where $p_u$ denotes a primitive generator of the centralizer of $u$.
Since $T$ generates $F_d$, the set $\{p_u\}_{u \in U}$ is a free basis of $F_d$.
For every $u\in U$, 
we define $T_u:=T\cap\form{ p_u}$.
Note that $T_u$ generates $\form{ p_u}$.
Thus, we obtain a partition
\[
T=\bigsqcup_{u \in U} T_u
\]
with the desired property.
\end{proof}

\medskip

We now prove Theorem \ref{thm:main}.

\begin{proof}[Proof of Theorem \ref{thm:main}]
Fix $d\ge 2$. 
By Proposition \ref{prop:cyclic} \eqref{p:12} and \eqref{p:13}, 
it suffices to prove that 
$ e_2(F_d)=\alpha_d(1,2)$ and $e_3(F_d)=\alpha_d(1,3)$.
Note that $\alpha_d(1,3) < 2d+1$. 

Fix a finite generating set $T$ of $F_d$.
We will prove that if
\begin{align}\label{eq:efd}
    2d-1 < e(F_d,T) < \alpha_d(1,3),
\end{align}
then $e(F_d,T)=\alpha_d(1,2)$. 
Note that $\alpha_d(1,3)=\beta_d$ by Proposition \ref{prop:cyclic} \eqref{p:b}.

Assume $T$ satisfies (\ref{eq:efd}).
We may assume $T$ is essentially reduced.
If $|T|=d$, then $T$ is a free basis and $e(F_d,T)=2d-1$. Hence, $|T|>d$.
Let us choose a minimal subset $U$ of $T$ 
whose image spans a rank-$d$ subgroup of the abelianization $F_d/[F_d, F_d]$.
Then we have $|U|=d$ and $U$ is a free basis of $\form{U}$.

By Lemma \ref{lem:splitting} and \eqref{eq:efd},
we have a decomposition $T=\bigsqcup_{u \in U}T_u$ where $\form{ T}$ is isomorphic to the free product $\ast_{u \in U}\form{ T_u}$ where each $\form{ T_u}$ is isomorphic to $\Z$.
Below we identify $\form{ T_u }$ with $\Z$ and assume that $T_u$ consists of positive integers, taking the inverses of elements if necessary.

By Proposition~\ref{lem:spherical-free} (\ref{p:free-prod2}), we have 
$$Q_{F_d,T}(x)=\sum_{u \in U} Q_{\mathbb  Z, T_u}(x)$$
for every $x\in [0, \rho_{F_d,T}).$
Since $\rho_{\mathbb{Z}, T_u}=1$ by Lemma \ref{lem:sph-cyc}, for $x \in [0, 1)$, we define 
$$\Phi(x)=\sum_{u \in U} Q_{\mathbb  Z, T_u}(x).$$

Since $T_u$ generates $\mathbb  Z$, we have $T_u=\{1\}$ if $|T_u|=1$.
Since $T$ is not a free basis of $F_d$, at least one of the sets $T_u$ is different from $\{1\}$.
We now distinguish three cases. 

First, suppose that  exactly one $T_u$ is $\{1, 2\}$ and  all the others are $\{1\}$. It is easy to see that 
$$T=\phi\left(\{a_1,a_1^2,a_2,\ldots,a_d\}\right)$$
for some $\phi\in\Aut(F_d)$.
Hence, $e(F_d,T)=\alpha_d(1,2)$ and we are done.

\bigskip
Next, suppose that exactly one $T_u$ is not $\{1\}$ and is not equal to $\{1, 2\}$, while $T_v=\{1\}$ for all $v\in U\setminus \{u\}$.
By Lemma \ref{lem:sph-cyc} \eqref{p:E2-2}, we have 
$$Q_{\mathbb{Z}, T_u}(x)  \ge  Q_{\mathbb{Z},\{1,3\}}(x), \quad x \in [0, 1).$$
Thus, by Example \ref{exmp:cyclic} \eqref{eq:13},
for $x \in [0,1)$, we have
\begin{equation}\label{eq:psix}
\Phi(x)=Q_{\mathbb{Z},T_u}(x) + (d-1)Q_{\mathbb{Z},\{1\}}(x)
\ge \frac{4x+2x^2}{1+3x+2x^2}+\frac{2(d-1)x}{1+x}=:\Phi_0(x).
\end{equation}
Observe that 
$$\Phi_0(x)=Q_{F_d,T_0}(x), \qquad x\in [0, \rho_{F_d,T_0}),$$ where
$$T_0=\{a_1,a_1^3,a_2,\dots, a_d\}.$$
Set
$$x_0:=\rho_{F_d,T_0}= \frac{1}{e(F_d,T_0)}=\frac{1}{\alpha_d(1,3)}.$$

By Lemma \ref{lem:spherical-basic} \eqref{p:hyp},
 $$\lim_{x\nearrow x_0}Q_{F_d,T_0}(x)=1.$$
Since $\Phi_0$ is continuous at $x_0$, this implies $\Phi_0(x_0)=1$, hence
$\Phi(x_0) \ge \Phi_0(x_0)=1$  by \eqref{eq:psix}.

Furthermore,
since $\Phi(0)=0$, 
the continuity of $\Phi$ gives  some $x_1\le x_0$ with $\Phi(x_1)=1$.
By Proposition~\ref{lem:spherical-free} \eqref{p:phi-strict},
$$e(F_d,T)=\frac{1}{x_1}\ge \frac{1}{x_0}=\alpha_d(1,3),$$
contradicting (\ref{eq:efd}).

\bigskip
Finally, suppose that there exist distinct $u, v \in U$ such that both $T_u$ and $T_{v}$ are not $\{1\}$.
Then, by Lemma \ref{lem:sph-cyc}, we have
$$Q_{\mathbb{Z}, T_u}(x), Q_{\mathbb{Z}, T_v}(x) \ge Q_{\mathbb{Z}, \{1,2\}}(x), \qquad
Q_{\mathbb{Z}, T_w}(x) \ge Q_{\mathbb{Z}, \{1\}}(x) \quad (w \not=u,v)$$
for all $x\in [0, 1)$.
Hence
\begin{equation}\label{eq:psixi}
\Phi(x) \ge 2Q_{\mathbb{Z}, \{1,2\}}(x) + (d-2)Q_{\mathbb{Z}, \{1\}}(x),
\qquad x\in [0,1).
\end{equation}

Since $\Phi_0(x)=Q_{\mathbb{Z}, \{1,3\}}(x)+(d-1)Q_{\mathbb{Z}, \{1\}}(x)$, we have 
$$\Phi(x)-\Phi_0(x)\ge 2Q_{\mathbb{Z}, \{1,2\}}(x) -Q_{\mathbb{Z}, \{1,3\}}(x) -Q_{\mathbb{Z},\{1\}}(x)=\frac{2x(1-x)}{(1+2x)(1+3x)}>0$$
for $x \in (0,1)$, where equality follows from
Example \ref{exmp:cyclic}.

Since $x_0 \in (0,1)$, we obtain 
$$\Phi(x_0) > \Phi_0(x_0)=1.$$
Now, since $\Phi(0)=0$ and $\Phi$ is continuous
on $[0,x_0]$ there exists $x_1 \in (0,x_0)$ such that $\Phi(x_1)=1$.
By Proposition \ref{lem:spherical-free} \eqref{p:phi-strict},
we have
$$e(F_d,T)=\frac{1}{x_1}> \frac{1}{x_0}=\alpha_d(1,3),$$
contradicting (\ref{eq:efd}).
This completes the proof. 
\end{proof}

We record the following corollary to characterize the equality case.

\begin{cor}\label{cor:cyc-eq}
(1) Suppose $e(F_d,T)=e_2(F_d)=\alpha_d(1,2)$. Then $T$ is,
up to  $\Aut(F_d)$, essentially equivalent to 
$$\{a_1,a_1^2,a_2,\ldots,a_d\}.$$

(2) Suppose $e(F_d,T)=e_3(F_d)=\alpha_d(1,3)$ and $T$ is
in the cyclic splitting case. Then  $T$ is,
up to $\Aut(F_d)$,
essentially equivalent to 
$$\{a_1,a_1^3,a_2,a_3,\ldots,a_d\}.$$

\end{cor}

\proof
(1) is clear. 
For (2), we write, up to essential equivalence,
$
T=\bigsqcup_{u\in U}T_u.
$
 If at least two $T_u$ are different from
$\{1\}$, the proof of Theorem~\ref{thm:main} gives
\[
e(F_d,T)>\alpha_d(1,3).
\]
Thus exactly one $T_u$ is different from $\{1\}$. The case
$T_u=\{1,2\}$ gives $e(F_d,T)=\alpha_d(1,2)$, while Lemma~\ref{lem:sph-cyc} \eqref{p:E2-2}
shows that every other possibility except $\{1,3\}$ gives  $e(F_d,T) > \alpha_d(1,3)$. Hence $T_u=\{1,3\}$.
\qed
 
\begin{rem}\label{rem:eq2}
To complete the characterization of the equality case,
 $$e(F_d,T)=e_3(F_d)=\alpha_d(1,3),$$ 
 it remains to consider the cyclically non-splitting factor case.
    For example, $$\Delta_d=\{a_1,a_1a_2,a_2,a_3,\dots,a_d\}$$
    satisfies $e(F_d,\Delta_d)=\beta_d=e_3(F_d)$.

    Assume that $T$ is essentially reduced and is in the cyclically
    non-splitting factor case, and satisfies the above equality. 
 Then the proof above gives 
 $V\subseteq T$ with
$$ |V|=d+1, \, \, \rk\form{V}=d$$
such that $V$ contains no commuting pair and  
    $$e(F_d,T)=e(\form{V},V)=\beta_d.$$
\end{rem}

\section{Equality cases for $e_2(F_d)$ and $e_3(F_d)$}\label{sec:eq}

In this section, we characterize the equality case in Proposition~\ref{prop:noncyc} and prove 
Proposition~\ref{prop:noncyc.eq}  and Theorem~\ref{thm:main.eq}. 
We divide the section into three subsections.
We first treat the rank-two case, then the
$(d+1)$-element case, and finally the effect of adjoining further generators,
 and complete the proof of Proposition~\ref{prop:noncyc.eq}. 

\subsection{The rank-two equality case without commuting pairs}

When $d=2$, the equality case splits according to whether the generating
set contains a commuting pair. The commuting case was determined in
Section~\ref{sec:gaps} and gives, up to essential equivalence,
$\{a,a^3,b\}$.

We now consider the case without commuting pairs and show that equality requires
$\Delta_2=\{a,b,ab\}$.

The main ingredient is the growth-gap theorem for quasi-convex subgroups
in hyperbolic groups. 
For a group $G$ generated by a finite generating set $T$
and for a subgroup $H$ of $G$, the \emph{ambient growth rate} of $H$ in $\cay(G,T)$ is defined by
$$e_{G,T}(H):=\limsup_{n\to\infty} \abs*{B_{G,T}(n)\cap H}^{1/n}.$$
We use the following theorem. 
For completeness, we also include a proof in the Appendix; Proposition~\ref{prop:ambient-hyp} treats the case of trivial finite radical, and Remark~\ref{rem:ambient-hyp}  explains the reduction of the general case to this one.

\begin{thm}[Dahmani-Futer-Wise, Theorem 1.1 in \cite{DFW2019}]\label{lem:dfw}
If $H$ is an infinite-index quasi-convex subgroup of a non-elementary word-hyperbolic group $G$ with a finite generating set $T$, 
then 
\[
e_{G, T}(H)<e(G, T).
\]
\end{thm}

We record the following immediate corollary.
\begin{cor}\label{lem:proper-free}
Let $T$ be a finite generating set of $F_d$.
If a subset $U\subseteq T$ generates a proper subgroup $H$ of $F_d$ such that $\rk H\le d$, then we have $e(H,U)<e(F_d,T)$.
\end{cor}
\begin{proof}
If $H\neq F_d$ and $H$ has finite index in $F_d$,
then $\rk H>d$ by Schreier's index formula.
Thus, by assumption, $H$ has infinite index in $F_d$.
Since every finitely generated subgroup of a free group is quasi-convex, 
by Theorem ~\ref{lem:dfw},
we conclude that
$$e(H,U) \le e_{F_d,T}(H) < e(F_d,T),$$
as required.
\end{proof}

We use the following standard
estimate for shortest coset representatives
in hyperbolic groups, see ~\cite[Lemma 1.105]{KS2024}.
\begin{lem}\label{lem:qc}
Fix $\delta, \kappa\ge 0$.
Let $G$ be a word-hyperbolic group with $\delta$-hyperbolic $\cay(G, T)$ for a finite generating set $T$, and let $H$ be a $\kappa$-quasi-convex subgroup of $G$.
Let 
$$\sigma: G/H \to G$$
be a \emph{shortest representative map},
i.e., for each $g$, we require 
that $\sigma(gH)\in gH$ and that 
$$|\sigma(gH)|_T =\min_{u\in gH} |u|_T.$$
Then there exists a constant $K\ge 0$ depending only on $\delta$ and $\kappa$ such that 
for all $g\in G$ we have
$$ |\sigma(gH)|_T + |\sigma(gH)^{-1}g|_T \le |g|_T+K.$$
\end{lem}

We record the following fundamental result on the \emph{pure} exponential growth of hyperbolic groups. 

\begin{thm}[\cite{Coornaert1993}]\label{lem:coornaert}
Let $G$ be a non-elementary word-hyperbolic group with a finite generating set $T$.
Then, there exist $c,C>0$ such that for all $n \ge 0$, 
\[
c \lambda^n \le |B_T(n)| \le C\lambda^n,
\]
where $\lambda=e(G,T)$.
\end{thm}

We now combine Theorem~\ref{lem:dfw}  and Lemma~\ref{lem:qc}  to obtain 
an exponential lower bound on the number of
cosets meeting a ball.

\begin{lem}[Coset growth]\label{lem:coset-count}
Let $G$ be a non-elementary word-hyperbolic group with a finite generating set $T$.
If $H$ is an infinite-index quasi-convex subgroup,
then there exists a constant $C>0$ such that 
$$
c_{G/H}(B_{G,T}(n))\ge C \cdot e(G,T)^n
$$
for all $n\ge 0$.
\end{lem}

\begin{proof}
Set $\lambda:=e(G,T)>1$.
By Theorem~\ref{lem:dfw}, there exist some $1\le \theta<\lambda $ and $C_1>0$ such that 
$$ | H\cap B_T(n)|\le C_1 \theta^n$$ for all $n$.
Applying Theorem~\ref{lem:coornaert} to
$(G,T)$,  we obtain some $0<c_2< C_2$ such that 
\begin{equation}\label{eq:pure-bn}
c_2 \cdot \lambda ^n \le b_n:=|B_T(n)| \le C_2 \cdot \lambda ^n.
\end{equation}
Let us fix a shortest-representative map
$\sigma:G/H\to G$ as in Lemma~\ref{lem:qc}.
By this lemma, there is a constant $K>0$ such that
for all $g\in G$ and for all $n>0$,
$$gH\cap B_{T}(n)
=\sigma(gH)H\cap B_{T}(n)
\subseteq \sigma(gH)\left( H\cap B_T\left(n+K-\abs*{\sigma(gH)}\right)
\right).
$$
Hence, we have some $C_3>0$ such that every $u\in G/H$ with $k=|\sigma(u)|_T$ satisfies
$$\abs*{u\cap B_T(n)}\le C_3\theta^{n-k}.$$
Define $r_k$ as the number of elements in the image of $\sigma$ having  $T$-length $k$.
In other words, we have
$$
r_k:=\abs*{\sigma(G/H)\cap S_T(k)}=\abs*{\sigma^{-1}(S_T(k))}.
$$
Considering the coset decomposition of $B_T(n)$, we obtain 
$$b_n
=
\abs*{
\bigcup_{k=0}^n
\left(
\bigcup_{u\in \sigma^{-1}(S_T(k))}
u\cap B_T(n)
\right)
}
\le C_3\sum_{k=0}^n r_k \theta^{n-k}.
$$
For all $0\le L\le n$, we have 
\[
\sum_{k=0}^{n-L}r_k \theta^{n-k}
\le C_2 \sum_{k=0}^{n-L}\lambda^k \theta^{n-k}
=C_2\lambda^n \sum_{k=0}^{n-L}\left(\frac{\theta}{\lambda}\right)^{n-k}
\le C' \left(\frac{\theta}{\lambda}\right)^L\lambda^n,
\]
where $C'=C_2/(1-\theta/\lambda)$.
Hence, for all $0< L\le n$, we have
\[
c_2 \lambda^n \le b_n \le C_3C'\left(\frac{\theta}{\lambda}\right)^L \lambda^n+C_3 \sum_{k=n-L+1}^n r_k \theta^{n-k}.
\]
Since $1\le \theta<\lambda$, there exist $\eta>0$ and $L$ such that for all $n\ge L$,
\[
\eta \lambda^n \le \left(c_2 -C_3C'\left(\frac{\theta}{\lambda}\right)^L\right)\lambda^n
\le C_3 \sum_{k=n-L+1}^n r_k \theta^{n-k}
\le C_3 \theta^{L-1}\sum_{k=n-L+1}^n r_k.
\]
Hence, for all $n\ge L$, we have
\[
c_{G/H}(B_T(n))=\sum_{k=0}^n r_k \ge \sum_{k=n-L+1}^n r_k \ge \frac{\eta}{C_3 \theta^{L-1}}\lambda^n.
\]
Letting 
$C:=\min\left(\eta C_3^{-1} \theta^{-L+1}
, \lambda^{-L}\right)$,
we obtain the claim.
\end{proof}

We now combine Lemma~\ref{lem:coset-count} with the expansion 
estimate in Proposition \ref{lem:nonsplit} to obtain strict inequality for Magnus markings.

\begin{lem}[Strict lower bound for Magnus markings]\label{lem:proper-strong}
Suppose $T$ is a Magnus marking of a non-elementary word-hyperbolic group $G$ such that  $t:=|T|-1\ge2$.
If $H=\form{T\setminus\{y\}}$ is an infinite-index quasi-convex subgroup of $G$ for some $y\in T$, then 
we have that $$e(G,T) > 1+2t-\frac{2}{t}.$$
\end{lem}
\begin{proof}
Let $b_n:=|B_T(n)|$ and $\theta_t:=1+2t-2/t$.
By Proposition~\ref{lem:nonsplit}, we have $e(G,T)\ge\theta_t$.
Assume $e(G,T)=\theta_t$ for contradiction.
Combining Proposition~\ref{lem:nonsplit} with Lemma~\ref{lem:coset-count}, we have some constant $C>0$ such that
$$b_{n}\ge \theta_t  b_{n-1}+\frac{2}{t}c_{G/H}(B_T(n-1))\ge \theta_t b_{n-1} +C \theta_t^{n-1}\ge\cdots\ge \theta_t^n+Cn\theta_t^{n-1}.$$
This contradicts the pure exponential growth of $G$ as given in Theorem~\ref{lem:coornaert}.
\end{proof}

We record two consequences of 
Lemma~\ref{lem:proper-strong}.

\begin{cor}[Hyperbolic one-relator groups]
\label{cor:one-relator-strict}
Let
\[
G=\form{ T\mid r}
\]
be a  word-hyperbolic one-relator group, where $r$ is
cyclically reduced and every element of $T$, or its inverse, appears in $r$.
Suppose that $|T|\ge 3$. Then
\[
 e(G,T)\ge 2|T|-1-\frac{2}{|T|-1},
\]
and if equality holds, then $G$ is virtually free. 
\end{cor}

\begin{proof}
By Magnus' Freiheitssatz, $T$ is a Magnus marking.
Hence, Proposition \ref{lem:nonsplit} gives 
\[
 e(G,T)\ge 1+2(|T|-1) -\frac{2}{|T|-1}=2|T|-1-\frac{2}{|T|-1},
\]

Now suppose that equality holds. 
For every $y\in T$,
the Magnus subgroup
\[
 H=\form{ T\setminus\{y\}}
\]
is quasi-convex in $G$ by  \cite[Theorem 4.7]{Linton}. Moreover, $H$ is a free group
of rank $|T|-1$. If $G$ were not virtually free, then $H$ would have infinite index in $G$, hence Lemma~\ref{lem:proper-strong} would give a strict inequality, a contradiction. 
\end{proof}

The next consequence states as follows. 

\begin{cor}\label{lem:proper}
Suppose $T=\{y_1,y_2,y_3\}$ is a Magnus marking of $F_2$ and set $H=\form{y_1,y_2}$. 
If $H\ne F_2$, then $e(F_2,T)>4$.
\end{cor}
\proof
Since $\rk F_2=\rk H=2$ and $H \not= F_2$,
$H$ has infinite index in $F_2$.
$H$ is quasi-convex since $F_2$ is a free group. Apply Lemma~\ref{lem:proper-strong} with $t=2$.
\qed


\medskip

We now obtain the rank-two rigidity 
that is the main result of this subsection.
Recall that, for $F_2=\form{a_1,a_2}$, 
$\Delta_2=\{a_1,a_2,a_1a_2\}$.

\begin{prop}[Rank-two rigidity]\label{lem:1-2-12}
Suppose $T=\{a,b,c\}$ is a Magnus marking of $F_2$. Then the following are equivalent.
\begin{enumerate}[(i)]
\item We have $c=(a^{\epsilon_1}b^{\epsilon_2})^{\epsilon_3}$ for some $\epsilon_i \in \{-1, 1\}$, $i=1, 2, 3$. Hence, $T$ is essentially equivalent to $\Delta_2$.
\item $P_{F_2,T} = 1+6\sum_{n\ge1}4^{n-1}X^n$;
\item $e(F_2,T)=4$.
\end{enumerate}
\end{prop}
\begin{proof}
The implication (i) $\Rightarrow$ (ii)
follows from Example~\ref{exmp:free}, 
and (ii) $\Rightarrow$ (iii) follows from Lemma \ref{lem:spherical-basic} \eqref{p:free-b-prod1}. 

Let us assume (iii). 
By Corollary~\ref{lem:proper}, every two-element subset of $T=\{a,b,c\}$ is a free basis for $F_2=\form{T}$.
Hence, there exists $\phi \in \Aut(F_2)$ such that $\phi(a)=a$ and $\phi(b)=c$. 
We claim that 
$$c=a^m b^{\epsilon}a^n$$ for some $m,n\in\mathbb{Z}$ and some $\epsilon \in \{-1, 1\}$.
Indeed, since $[a],[c]$ form a basis of $H_1(F_2;\mathbb Z)$, we may write
\[
[c]=p[a]+\epsilon[b],\qquad \epsilon\in \{-1, 1\}.
\]
Define $\varphi \in \Aut(F_2)$ by $\varphi(a)=a$ and $\varphi(b)=a^p b^\epsilon$. Then
$\varphi^{-1}\phi$ acts trivially on $H_1(F_2;\mathbb Z)$, and hence is inner
by Nielsen's theorem,
which states that
\[
\ker\bigl(\operatorname{Aut}(F_2)\to GL(2,\mathbb Z)\bigr)
=\operatorname{Inn}(F_2).
\]
Since $\varphi^{-1}\phi$ fixes $a$, it is a conjugation with a power of $a$, which proves the claim.

Similarly, there is $\psi\in\Aut(F_2)$ such that
$\psi(b)=b$ and $\psi(a)=c$, which implies that
$c=b^r a^{\delta}b^s$ for some $r,s\in\mathbb{Z}$ and for some $\delta \in \{-1, 1\}$.

Combining those two expressions and using the uniqueness of reduced words in the free basis $\{a,b\}$, we obtain
$$c=(a^{\epsilon_1}b^{\epsilon_2})^{\epsilon_3}$$
for some $\epsilon_i \in \{-1, 1\}$, $i=1, 2, 3$.
This completes the proof of the implication (iii) $\Rightarrow$ (i).
\end{proof}

Recall our convention that 
$A_d=\{a_1,a_2,\ldots,a_d\}$ denotes a free basis of $F_d$,
and that 
$$\Delta_d:=\{a_1,a_2,a_1a_2,a_3,\ldots,a_d\}.$$

We now show that adding a new element 
to $\Delta_2$ produces a non-standard three-element generating set without commuting pairs.

\begin{lem}\label{lem:non-triangle}
If $s\not\in \Delta_2\cup \Delta_2^{-1}\cup\{1\}$ is an element of $F_2$,
then there exists a three-element subset $R\subseteq \Delta_2\cup\{s\}$ which generates $F_2=\form{A_2}$, contains no commuting pair, and is not essentially equivalent to $\Delta_2$  modulo $\Aut(F_2)$.

In particular, such an $R$ is a Magnus marking of $F_2$.
\end{lem}

\begin{proof}
If $s=u^m$ for some $u\in \Delta_2$ with $|m|\ge2$,
then we may simply define
$$R:=\left(\Delta_2\setminus\{u\}\right)\cup\{u^m\}.$$
Thus, we may assume that the set $\Delta_2\cup\{s\}$ has no commuting pair. If $\{a_1,a_2,s\}$ is not essentially equivalent to $\Delta_2$  modulo $\Aut(F_2)$, then we take $R:=\{a_1,a_2,s\} $ and we are done. 

Otherwise, 
$\{a_1,a_2,s\}$ is essentially equivalent to $\Delta_2$
modulo $\operatorname{Aut}(F_2)$. Then
\[
[s]+[a_1]+[a_2]=0
\quad\text{in }H_1(F_2;\mathbb Z/2).
\]
Here, we have used the fact that $\Aut(F_2)$ acts on the homology classes by preserving the zero class.
Now, choose $R:=\{a_1,a_1a_2,s\}$.
Since
\[
[a_1]+[a_1a_2]+[s]=[a_1]\neq0,
\]
and the sum of the three homology classes of $\Delta_2$ is zero, $R$
is not essentially equivalent to $\Delta_2$ modulo $\operatorname{Aut}(F_2)$.
\end{proof}

\begin{cor}[No extra generator in rank two]\label{cor:no-extra-rank2}
    Let $s \in F_2$. If 
    \[
    e(F_2, \Delta_2 \cup \{s\})=4,
    \]
    then $s \in \Delta_2\cup\Delta_2^{-1} \cup \{1\}$, i.e., $\Delta_2 \cup \{s\}$ is essentially equivalent to $\Delta_2$.
\end{cor}
\proof
Suppose $s \not\in \Delta_2\cup\Delta_2^{-1} \cup \{1\}$. By Lemma~\ref{lem:non-triangle}, there exists a 
Magnus marking $R \subset \Delta_2 \cup \{s\}$ of $F_2$ which  is not essentially equivalent to $\Delta_2$.

By Proposition~\ref{prop:noncyc} and the monotonicity of growth rates, 
$$ 4 =\beta_2 \le e(F_2,R) \le e(F_2, \Delta_2 \cup \{s\})=4.$$
Hence $e(F_2,R)=4$. Proposition~\ref{lem:1-2-12}  then implies that $R$ is essentially equivalent to 
$\Delta_2$, a contradiction. 
\qed

\medskip

Combining Proposition~\ref{lem:1-2-12} and Corollary~\ref{cor:no-extra-rank2} with 
the cyclically splitting equality case established in Section 4, we obtain the 
complete classification of the equality case in rank two.
Namely, if $T$ is a finite generating set of $F_2$ with $e(F_2,T)=4$, then, up to
$\Aut(F_2)$, $T$ is essentially equivalent to either $\{a_1,a_1^3,a_2\}$ or $\Delta_2$.
This proves Theorem~\ref{thm:main.eq} (2) for $d=2$.

\subsection{The equality case for $(d+1)$-element generating sets of $F_d$}

The next lemma shows that a non-standard rank-two Magnus marking remains strictly above $\beta_d$ after adjoining the free basis $A_{d-2}$.

\begin{lem}\label{lem:r0}
Let $d\ge2$. 
Let $R$ be a three-element Magnus marking of $F_2$ that is not essentially equivalent to $\Delta_2$  modulo $\Aut(F_2)$.
Then we have
$$e(F_2\ast F_{d-2},R\sqcup A_{d-2})>\beta_d.$$
\end{lem}
\begin{proof}
Suppose first that $d=2$. 
By Proposition \ref{lem:nonsplit},
$e(F_2,R) \ge4$, while
Proposition~\ref{lem:1-2-12} shows that equality would force $R$ to be 
essentially equivalent to $\Delta_2$. Hence $e(F_2,R) > 4=\beta_2.$ 

Suppose that $d \ge 3$.
Let $\rho_0$ be the radius of convergence of $P_{F_2\ast F_{d-2}, R\sqcup A_{d-2}}$.
Set
$$\rho_{\min}:=\min(\rho_{F_2,R},\rho_{F_{d-2},A_{d-2}}).$$
We define
$$\Phi(x):=Q_{F_2,R}(x)+Q_{F_{d-2},A_{d-2}}(x) $$
for $x\in[0,\rho_{\min})$.
Furthermore, 
we define 
$$\Phi_0(x):= \frac{6x}{1+2x}+\frac{2(d-2)x}{1+x}$$
for $x \in [0,\infty).$
By Proposition~\ref{lem:nonsplit} and Example~\ref{exmp:free} \eqref{eq:d},
$$\Phi(x) \ge \Phi_0(x)$$
for $x \in [0,\rho_{\min}).$

Moreover, by Proposition \ref{lem:spherical-free} \eqref{p:free-prod2},
$$\Phi_0(x)=Q_{F_d,\Delta_d}(x),$$
for $x \in [0,x_0)$,
where $x_0:=\rho_{F_d,\Delta_d}$.
By Proposition \ref{lem:spherical-free} \eqref{p:phi-strict}, 
we have
$$\Phi_0(x_0)=1.$$

Propositions \ref{lem:1-2-12} and \ref{lem:nonsplit} show that $e(F_2, R)>4$ and for every $x \in (0, \rho_{F_2, R})$,
\[
Q_{F_2, R}(x)>\frac{6x}{1+2x}.
\]
Thus, for all $x\in (0, \rho_{\min})$,
we have $\Phi(x)>\Phi_0(x)$.
Since $\rho_0<\rho_{\min}$ by Proposition \ref{lem:spherical-free} \eqref{p:free-prod1},
by Proposition \ref{lem:spherical-free} \eqref{p:phi-strict},
we have
\[
1=\Phi(\rho_0)>\Phi_0(\rho_0).
\]
Noting that $\Phi_0$ is strictly increasing on $[0, \infty)$,
we have $\rho_0<x_0$, i.e.,
\[
e(F_2\ast F_{d-2}, R\sqcup A_{d-2})>\frac{1}{x_0}=\beta_d.
\]
The proof is complete. 
\end{proof}

\begin{lem}[Rigidity of the equality case]\label{lem:noncyc-eq_new}
    Under the hypotheses of Proposition~\ref{prop:noncyc}, suppose
    further that $e(F_d,U)=\beta_d$.
    Let $W\subset U$ be an inclusion-minimal set subject to
    $\rk \form{ W} < |W|$.
    Put $H =\form{ W }$. 
    
    Then $W$ is essentially equivalent to $\Delta_2$ modulo
    $\Aut(H)$,
    and $F_d=H \ast \form{ U \backslash W }$,
    where $U \backslash W$ is a free basis of the second factor.
    Consequently, $U$ is essentially equivalent to $\Delta_d$ modulo
    $\Aut(F_d)$.
\end{lem}

\proof
By the inclusion-minimality of $W$, the set $W$ is a Magnus marking
of $H$, as observed in the proof of Proposition~\ref{prop:noncyc}.
By Remark~\ref{rem:concyc-eq}, 
$$\rk H=2, |W|=3, e(H,W)=4,$$
and $$F_d=H*\form{ U \backslash W },$$
where $U \backslash W$ is a free basis of the second factor.

Proposition~\ref{lem:1-2-12}  gives an isomorphism
$$\phi_H: H \to \form{ a_1, a_2 }$$
sending 
 $W$ to a set  essentially equivalent to $\Delta_2$.
 
 Since $U \backslash W$ is a free basis of $\form{ U \backslash W }$, there is an isomorphism
 $$\phi': \form{ U \backslash W } \to \form{ a_3, \dots, a_d}$$
 sending $U \backslash W$ to $\{a_3, \dots, a_d\}$.
  The free-product isomorphism $\phi_H * \phi'$ is 
an automorphism of $F_d$ sending  $U=W \sqcup (U\backslash W)$ to a set essentially 
equivalent to $\Delta_2 \sqcup \{a_3, \dots, a_d\}=\Delta_d$.
This proves the lemma.
\qed

\medskip

Lemma~\ref{lem:noncyc-eq_new}  therefore shows 
that if $U$ is a $(d+1)$-element generating
set of $F_d$ containing no commuting pairs and $e(F_d,U)=\beta_d$, then $U$
is essentially equivalent to $\Delta_d$ modulo $\Aut(F_d)$.

We next record a consequence of 
Remark~\ref{rem:concyc-eq} for a generating set
obtained by adjoining one element to a free basis.
This will be used in the next subsection. 

\begin{lem}[Equality after adjoining one generator]\label{lem:no-extra-basis}
Let $y\not\in A_d\cup A_d^{-1}\cup\{1\}$ such that
$A_d\cup\{y\}$ contains no commuting pairs.
If
$$e(F_d,A_d\cup\{y\})=\beta_d,$$
then we can write $y=a_i^\epsilon a_j^\delta$ for some $\epsilon,\delta\in\{-1, 1\}$ and $i\ne j$.
\end{lem}

\begin{proof}
Set $U:=A_d\cup\{y\}$. Then $U$ satisfies the hypotheses of 
Proposition~\ref{prop:noncyc}.
Choose an inclusion-minimal set $W \subset U$ subject to 
$\rk \form{ W } <|W|$, and put $H=\form{ W }$.
Since $A_d$ is a free basis, we have $y \in W$.
By Remark~\ref{rem:concyc-eq}, we have 
$\rk H =2, |W|=3, e(H,W)=4$. 
Hence 
$$W=\{a_i,a_j,y\}$$
for some distinct $i,j$.
By the inclusion-minimality of $W$, it is a Magnus marking of $H \simeq F_2$.

Applying Proposition \ref{lem:1-2-12} to $W$, we have 
$$y=a_i ^\epsilon  a_j ^\delta$$
after possibly interchanging $i$ and $j$, where
$\epsilon,\delta \in \{-1, 1\}$. This proves the lemma.
\end{proof}

This completes the analysis of the equality case for generating 
sets of cardinality $d+1$.

\subsection{Rigidity under adjoining generators}

The last step toward the proof of Proposition~\ref{prop:noncyc.eq} is to consider the special case that $V=\Delta_d$ and $T\setminus V$ is a single element, as below.

\begin{lem}\label{lem:no-extra-triangle}
If we have $$e(F_d, \Delta_d\cup\{s\})=\beta_d$$ for some $s\in F_d$,  
then $s \in \Delta_d\cup \Delta_d^{-1}\cup\{1\}$.
Consequently, 
$\Delta_d\cup \{s\}$ is essentially equivalent to $\Delta_d$.
\end{lem}

\begin{proof}
For $d=2$, the assertion follows from Corollary~\ref{cor:no-extra-rank2}.
Hence assume $d\ge3$.

Assume, for contradiction,  that
$$s\not\in \Delta_d\cup \Delta_d^{-1}\cup\{1\}.$$
We distinguish cases according to whether $\Delta_d \cup \{s\}$ contains
a commuting pair, or not. 
In the case where the set contains a commuting pair, we have either $s=a_k^m$ for some $k\ge 3$ and for some integer $m$, or $s \in \form{ a_1,a_2}\setminus(\Delta_d\cup\Delta_d^{-1})$. 
We treat these cases first, and then consider the case where the set contains no commuting pair.

\bigskip
\noindent
\textbf{Case 1}. $s=a_k^m$ for some $k\ge3$ and for some integer $m$.

After permuting $a_3, \dots, a_d$, we may assume $k=3$. By assumption $|m|\ge2$. We may assume $m>0$ by changing $a_k$ to $a_k^{-1}$ if necessary.

Since $\Delta_d \cup \{ s\} =\Delta_2  \sqcup \{a_3,a_3^m\} \sqcup \{a_4, \dots, a_d\}$, we have 
$$F_d=\form{ a_1,a_2 } *\form{ a_3 } * \form{ a_4 } * \dots *\form{ a_d}.$$
Recall that $\rho_{F_2,\Delta_2}=1/4$. We define
for $0 \le x < 1/4$
$$\Phi(x)= Q_{F_2,\Delta_2}(x) +Q_{\mathbb{Z}, \{1,m\}}(x)+ (d-3)Q_{\mathbb{Z}, \{1\}}(x) .$$
By Examples \ref{exmp:free} \eqref{eq:triangle} and \ref{exmp:cyclic} \eqref{eq:1} \eqref{eq:12}, by Lemma~\ref{lem:sph-cyc} \eqref{p:E2-1}, we have for $0< x <1/4$
$$\Phi(x)
\ge \frac{6x}{1+2x}+\frac{4x}{1+3x}+\frac{(2d-6)x}{1+x}
.$$

Set $x_0:=1/\beta_d <1/4$.
Evaluating the preceding inequality at $x_0$, we obtain 
\begin{align*}
\Phi(x_0)
&
\ge \frac{6x_0}{1+2x_0}+\frac{4x_0}{1+3x_0}+\frac{(2d-6)x_0}{1+x_0}
\\
&> \frac{6x_0}{1+2x_0}+\frac{(2d-4)x_0}{1+x_0}=1.
\end{align*}
Indeed, the difference between the two sides of the strict inequality is
\[
\frac{4x_0}{1+3x_0}-\frac{2x_0}{1+x_0}
=
\frac{2x_0(1-x_0)}{(1+x_0)(1+3x_0)}>0,
\]
and the last equality follows from substituting 
$x_0=1/\beta_d$ into the expression and from using the equation
$\beta_d^2-(2d-1)\beta_d-(4d-4)=0$.

Since $\Phi(0)=0$ and $\Phi$ is continuous and strictly 
increasing on $[0,1/4)$, there is a unique
$\rho \in (0,x_0)$ such that  $\Phi(\rho)=1$.

By Corollary~\ref{cor:spherical-free-finite}, $\rho$ is the radius of convergence of $P_{F_d, \Delta_d \cup \{s\}}$. Hence
$$e(F_d, \Delta_d \cup \{s\}) = \frac{1}{\rho} > \frac{1}{x_0} =\beta_d,$$
a contradiction. This case is complete.


\bigskip
\noindent
\textbf{Case 2}. $s\in\form{a_1,a_2}\setminus(\Delta_d\cup\Delta_d^{-1})$.
This is the second possibility in the commuting-pair case.  

We set $R_1:=\Delta_2\cup\{s\}$ and $R_2:=\{a_3,\ldots,a_d\}$.
By Lemma~\ref{lem:non-triangle}, $R_1$ contains a three-element subset $R$ that generates $F_2=\form{a_1,a_2}$,
contains no commuting pair, and 
is not essentially equivalent to $\Delta_2$ modulo $\Aut(F_2)$.

Hence Lemma~\ref{lem:r0} applies to $R\cup R_2$, and 
$$e(F_d,\Delta_d\cup\{s\})\ge e(F_d,R\cup R_2)>\beta_d,$$
a contradiction. This case is complete.

\bigskip
Cases 1 and 2 complete the commuting-pair case. It remains
to consider the case with no commuting pair. 

\noindent
\textbf{Case 3}. $\Delta_d\cup\{s\}$ has no commuting pair. 

Then $A_d \cup \{s\}$ contains no commuting pair either. 
By Proposition~\ref{prop:noncyc} applied to $A_d\cup\{s\}$, 
$$\beta_d \le e(F_d, A_d \cup \{s\}) \le e(F_d, \Delta_d \cup \{s\}) =\beta_d.$$
Therefore, 
$$\beta_d=e(F_d,\Delta_d\cup\{s\})=e(F_d,A_d\cup\{s\}).$$

By Lemma~\ref{lem:no-extra-basis}, 
we have
$s=a_i^{\epsilon_1}a_j^{\epsilon_2}$ for some $i\ne j$ and some $\epsilon_1, \epsilon_2 \in \{-1, 1\}$.
If $i,j \in \{1,2\}$, then $s \in \form{ a_1,a_2 }$, which
was already treated in Case 2. 
We may therefore assume that  at least one of $i,j$
is not in $\{1,2\}$. 

We first consider the case $i,j \not\in \{1,2\}$. 
Then $d \ge 4$. After permuting $a_3,\dots, a_d$ and replacing $a_3$ and $a_4$ by their inverses if necessary, we may assume that 
$i=3, j=4$ and $s=a_3a_4$.
Thus $\{a_3,a_4,s\}$ is a copy of $\Delta_2$ in the 
free factor $\form{ a_3,a_4 }$. 

We have the free-product decomposition
$$F_d=\form{ a_3, a_4} * \form{ a_1,a_2, a_5,\dots, a_d }.$$
The generating set accordingly decomposes as
$$\Delta_d \cup \{ s\}=\{a_3,a_4,a_3a_4\} \sqcup \{a_1,a_2,a_1a_2,a_5,\dots, a_d\},$$
which corresponds to $\Delta_2 \sqcup \Delta_{d-2}$.

For $0\le x<\rho_{\min}:=\min(1/\beta_2,1/\beta_{d-2})$, we define
$$\Phi(x):=Q_{F_2,\Delta_2}(x) + Q_{F_{d-2},\Delta_{d-2}}(x)=\frac{12x}{1+2x} +\frac{(2d-8)x}{1+x}.$$
Set $$x_0=\frac{1}{\beta_d}.$$
Since $$x_0 =\frac{1}{\beta_d} < \rho_{\min},$$
 $\Phi$ is defined and 
continuous at $x_0$. 

$$\Phi(x_0)=\frac{12x_0}{1+2x_0} +\frac{(2d-8)x_0}{1+x_0} =1+\frac{2x_0(1-x_0)}{(1+x_0)(1+2x_0)} >1,$$
where the final equality follows from 
$x_0=1/\beta_d$.

 Since $\Phi(0)=0$, $\Phi(x_0)>1$, and $\Phi$ is continuous and strictly 
 increasing on $[0,\rho_{\min})$, the equation 
 $\Phi(x)=1$ has a unique solution $\rho \in (0,x_0) \subset (0, \rho_{\min})$. 
 
 By Proposition \ref{lem:spherical-free}, 
 $\rho$ is the radius of convergence  of $P_{F_d, \Delta_d \cup \{ s\}}$. Therefore, 
 $$e(F_d, \Delta_d \cup \{s\})=\frac{1}{\rho}.$$
 
 Since $\rho <x_0$, we have 
$$e(F_d,\Delta_d\cup\{s\}) =\frac{1}{\rho} >\frac{1}{x_0} =\beta_d,$$
a contradiction.

\medskip
Lastly, we consider the case where $i\in\{1,2\}$ and $j\not\in\{1,2\}$.
As before, without loss of generality,
we may assume $i=1$ and $j=3$, so that 
$s=a_1 a_3$. 

Set 
$$A:=\{a_1a_2,a_2,a_3,a_4,\ldots,a_d\}.$$
Then, $A$ is a free basis of $F_d$, and 
$$A \cup \{s\} \subset \Delta_d \cup \{s\}.$$
By assumption, $A \cup \{s\}$ contains no commuting pair. 
Proposition~\ref{prop:noncyc} implies
$$\beta_d= e(F_d,\Delta_d\cup\{s\})\ge e(F_d,A\cup\{s\})\ge\beta_d.$$
Hence $e(F_d, A \cup \{s\})=\beta_d$. 
Since $s=a_1a_3\notin A\cup A^{-1}$,
by Lemma~\ref{lem:no-extra-basis}, we have
$s=u^{\epsilon_1}v^{\epsilon_2}$ for some $u,v\in A$ and some $\epsilon_1, \epsilon_2 \in \{-1, 1\}$.

Passing to $H_1(F_d,\mathbb{Z}/2)$, we have 
$[s]=[a_1]+[a_3]$.
We claim that $[a]+[b]+[s]\ne0$ for all distinct $a,b\in A$.
Indeed, otherwise, $[a_1]+[a_3]=[a]+[b]$ for some distinct $a, b \in A$.
Among the elements in $A$, only $a_1a_2$ has non-zero 
$a_1$-coordinate. Hence one of $a,b$ must be $a_1a_2$.
Then the other must have homology class
$[a_1]+[a_3]+[a_1a_2]=[a_2]+[a_3]$,
whereas no element in $A$ has 
this homology class, proving the claim. 
However, this equality $s=u^{\epsilon_1} v^{\epsilon_2}$ implies that $[u]+[v]+[s]=0$ for distinct $u,v \in A$, contradicting the claim. 

This exhausts all cases and completes the proof. 
\end{proof}

We prove Proposition~\ref{prop:noncyc.eq}.
\begin{proof}[Proof of Proposition \ref{prop:noncyc.eq}]
Since $V$ satisfies the hypothesis of Proposition~\ref{prop:noncyc}, and 
$e(F_d,V)=\beta_d$,
Lemma \ref{lem:noncyc-eq_new} shows that  $V$ is, up to $\Aut(F_d)$,
essentially equivalent to $\Delta_d$.
After applying an automorphism, we may therefore assume that 
$V\cup V^{-1}=\Delta_d \cup \Delta_d^{-1}$. 

Let $s \in T$. Since $V \subset T$, we have 
$$\beta_d=e(F_d,V) \le e(F_d, V \cup \{s\}) \le e(F_d,T)=\beta_d.$$
Hence $e(F_d, V \cup \{s\})=\beta_d$.
Since $V\cup V^{-1}=\Delta_d \cup \Delta_d^{-1}$, this gives
$$e(F_d, \Delta_d\cup \{s\})=\beta_d.$$
By Lemma \ref{lem:no-extra-triangle}, $s \in \Delta_d \cup \Delta_d^{-1}\cup\{1\}$.
Since this holds for every $s \in T$, the generating set $T$
is essentially equivalent to $\Delta_d$.
This proves the proposition. 
\end{proof}

\medskip

We prove Theorem \ref{thm:main.eq}.

\begin{proof}[Proof of Theorem \ref{thm:main.eq}]
The ``if" directions follow from Theorem \ref{thm:main}.
We prove the ``only if" directions.
Suppose first that 
$e(F_d,T)=e_2(F_d)=\alpha_d(1,2)$.
By Corollary \ref{cor:cyc-eq} (1), $T$ is, up to $\Aut(F_d)$, 
essentially equivalent to 
$$\{a_1, a_1^2, a_2, a_3, \dots,a_d\}.$$
This proves (1). 

Suppose next that $e(F_d,T)=e_3(F_d)=\alpha_d(1,3)=\beta_d.$
Up to essential equivalence, we may assume that $T$ is essentially 
reduced. 
If $T$ is in the cyclically splitting case, Corollary \ref{cor:cyc-eq} (2) shows
that $T$ is, up to $\Aut(F_d)$, essentially equivalent to
$$\{a_1, a_1^3, a_2, a_3, \dots, a_d\}$$
and we are done. 

We may therefore assume that $T$ is in the cyclically non-splitting factor case.
By Remark \ref{rem:eq2}, there exists $V \subset T$ such that 
$$|V|=d+1, \quad \rk \form{ V } =d,$$
$V$ contains no commuting pair, and 
$$e(F_d,T)=e(\form{ V }, V)=\beta_d.$$
We claim that $\form{ V } =F_d$. Otherwise, Corollary \ref{lem:proper-free} gives 
$$e(\form{ V }, V) < e(F_d,T)=\beta_d,$$
contradicting the preceding equality. Hence $\form{ V } =F_d$, so that 
$$e(F_d, V)=e(F_d,T)=\beta_d.$$

Proposition \ref{prop:noncyc.eq} now implies that $T$ is, up to 
$\Aut (F_d)$, essentially equivalent to $\Delta_d$.
This proves (2), and hence Theorem  \ref{thm:main.eq}.
\end{proof}

\section{Proof of Theorem \ref{thm:surface}}\label{sec:surface}

We prove Theorem \ref{thm:surface} in this section.
To this end, we use the following lemma, which may be of independent interest. We will give several other applications of this lemma.

\begin{lem}\label{lem:surface}
For every finite generating set $T$ of $\Gamma_g$,
if $e(\Gamma_g,T)<4g-1$,
then $T\cup T^{-1}$ contains a $2g$-element set $X$ such that $\Gamma_g=\form{ X}$.

Furthermore, if $Y$ is a $2g$-element set such that $\Gamma_g=\form{ Y}$,
then every proper subset $Z$ in $Y$ is a free basis of $\form{ Z}$.
\end{lem}

\begin{proof}
Define $X:=\{x_1, \dots, x_{2g}\}$ in $T\cup T^{-1}$ whose images in the abelianization of $\Gamma_g$ form a basis over $\mathbb{Q}$.
Set $H:=\form{ X}$.
We will show that $H=\Gamma_g$.

If $H\neq \Gamma_g$ and $H$ has finite index $N \ge 2$ in $\Gamma_g$,
then the rank of $H/[H, H]$ is equal to $2+2N(g-1)$.
This follows since $H$ is the fundamental group of an $N$-fold cover $\widetilde{\Sigma}$ of the surface $\Sigma_g$, and the Euler characteristic of $\widetilde{\Sigma}$ is $N$ times that of $\Sigma_g$.
For $N\ge 2$,
we have 
\[
2+2N(g-1)>2g.
\]
This is absurd since $H$ is generated by $2g$ elements.
Thus, if $H\neq \Gamma_g$,
then $H$ has infinite index in $\Gamma_g$.
In this case, $H$ is the fundamental group of a non-compact surface, and thus $H$ is a free group.
Since the images of elements in $X$ are independent in $\Gamma_g/[\Gamma_g, \Gamma_g]$ over $\mathbb{Q}$,
they are independent in $H/[H, H]$ over $\mathbb{Q}$.
Hence $X$ defines a free basis of $H$.
Then, we have
\[
e(\Gamma_g,T) \ge e(H, X) =4g-1,
\]
contradicting the assumption that $e(\Gamma_g,T)<4g-1$.
Thus, $\Gamma_g=H=\form{ X}$.
This shows the first claim.

We show the second claim.
For $Z \subset Y$ with $|Z|=m \le 2g-1$, 
consider $\form{Z}$.
An analogous argument to the above shows that $\form{Z}$ has infinite index in $\Gamma_g$, and that $\form{Z}$ is a free group.
Furthermore, considering the abelianization of $\form{Z}$, the rank of $\form{Z}$ is $m$, and thus $Z$ forms a free basis of $\form{Z}$.
\end{proof}

\begin{proof}[Proof of Theorem \ref{thm:surface}]
By~\cite{FS}, we have a finite generating set $T$ of $\Gamma_g$ such that
$$e(\Gamma_g)=e(\Gamma_g,T).$$
We assume that $e(\Gamma_g,T)<4g-1$, otherwise we have nothing to prove.
By Lemma \ref{lem:surface}, there exists a $2g$ element set $X$ in $T\cup T^{-1}$ that is a Magnus marking of  $\Gamma_g$.
Thus, by Proposition \ref{lem:nonsplit}, we have
\[
e(\Gamma_g,T)\ge e(\Gamma_g, X)\ge 2\left(2g-1-\frac{1}{2g-1}\right)+1=4g-1-\frac{2}{2g-1}.
\]
In fact, the last inequality is strict.
Indeed,
for an arbitrary $x\in X$, the group
$$H:=\form{X\setminus\{x\}}$$
is a rank-$(2g-1)$ free group, which is of infinite-index in $\Gamma_g$. 
Since every finitely generated subgroup of $\Gamma_g$ is quasi-convex~\cite[Proposition 2]{Pittet1993}, so is $H$.
Thus, Lemma~\ref{lem:proper-strong} yields a strict inequality, and the theorem is proved. 
\end{proof}

\begin{rem}\label{rem:surface}
The spherical growth series  of $\Gamma_g$, with respect to the standard generating set $T_g$ for $g\ge 2$,
was computed in, for example, \cite[VI.A.8]{dlHarpe2000}:
$$P_{\Gamma_g, T_g}(x)=\frac{(x+1)(x^{2g}-1)}
{x^{2g+1}-(4g-1)x^{2g}+(4g-1)x-1}.$$

Let
$$D_g(x):=x^{2g+1}-(4g-1)x^{2g}+(4g-1)x-1.$$
By Lemma~\ref{lem:spherical-basic} \eqref{p:free-b-prod1},
$e(\Gamma_g,T_g)^{-1}$ is the radius of convergence of
$P_{\Gamma_g,T_g}$, and hence the smallest positive real root of
$D_g(x)=0$.
Since $D_g(x)$ is  an anti-reciprocal polynomial, $e(\Gamma_g,T_g)$
is the largest positive real root of $D_g(x)=0$.

A direct computation shows 
$$D_g\left(4g-1-\frac{2}{2g-1}\right)<0.$$
It follows that 
$$e(\Gamma_g, T_g) > 4g-1 -\frac{2}{2g-1}.$$
Hence, Theorem \ref{thm:surface} is not 
sufficient to conclude that $T_g$ realizes $e(\Gamma_g)$.
\end{rem}

We give several corollaries of Lemma \ref{lem:surface}.

{\begin{cor}[see Corollary~\ref{cor:min-gen-intro}]\label{cor:min-gen}
    Let $g \ge2$, and let $T$ be a finite generating set of $\Gamma_g$ such that  $e(\Gamma_g)=e(\Gamma_g,T)$.
    Then, $T \cup T^{-1}$ contains a subset $X$ with $|X|=2g$ 
    such that $\form{ X } =\Gamma_g$ and 
    $e(\Gamma_g, X)=e(\Gamma_g)$.

In particular, $e(\Gamma_g)$ is realized
    by a $2g$-element generating set. 
\end{cor}}

\proof
Let $T$ be a generating set with $e(\Gamma_g)=e(\Gamma_g,T)$.
Since $e(\Gamma_g,T_g)< 4g-1$, we have 
$e(\Gamma_g,T) < 4g-1$.
By Lemma \ref{lem:surface}, there exists
a $2g$-element set $X$ in $T\cup T^{-1}$ such that 
$\form{ X }=\Gamma_g$.
Since $e(\form{ X }, X) \le e(\Gamma_g,T)$, 
we have $$e(\Gamma_g,X)=e(\form{ X }, X) = e(\Gamma_g,T)=e(\Gamma_g).$$

The second statement follows from the first.
\qed

\medskip

In \cite{Los},
a presentation of a surface group is called {\it geometric}
if its two-dimensional Cayley complex is planar, i.e., homeomorphic to the plane
in this setting. 
We call a finite generating set $T$ of $\Gamma_g$ {\it geometric} if there exists a geometric presentation of $\Gamma_g$ whose set of generators is $T$.

In the following argument, without loss of generality for our purpose, if a presentation 
contains a relator of the form $t=s^{\pm 1}$, we remove the relator and one of the generators $s$ and $t$.  
Consequently, the presentation complex contains no 
bigons. Similarly we may assume it contains no monogons. 

Recall that the two-dimensional Cayley complex is the universal
cover of the presentation complex. Hence, if a presentation of $\Gamma_g$ is geometric,
then the presentation complex is homeomorphic to $\Sigma_g$.
Since a presentation complex has 
a single vertex and contains no monogons or bigons, Euler's formula gives $|T|\le 6g-3$.
Hence there are only finitely many such presentation complexes up to
homeomorphism of $\Sigma_g$.
Consequently, there are only finitely many geometric 
presentations of $\Gamma_g$, up to $\Aut(\Gamma_g)$.

For example, the standard presentation of $\Gamma_g$ is geometric, hence $T_g$ is geometric.
Let $e_{geo}(\Gamma_g)$ be the smallest growth rate
among the finite generating sets of $\Gamma_g$ for
geometric presentations.
Note that $e_{geo}(\Gamma_g)$ is realised. 
Also, $e_{geo}(\Gamma_g) < 4g-1$ since $e(\Gamma_g, T_g) < 4g-1$.
Note that we do not know if $e_{geo}(\Gamma_g)=e(\Gamma_g)$.

\begin{lem}[Geometric generating sets of small growth rates]\label{lem:geo}
Suppose $T$ is a finite generating set 
of $\Gamma_g$. Suppose $T$ is geometric with 
respect to a presentation $P_T$.

If $e(\Gamma_g, T) <4g-1$, then there exists a generating set $X$
of $\Gamma_g$ with $|X|=2g$, $X \subset T \cup T^{-1}$ such that $X$ is geometric with respect to a one-relator presentation $P_X$, whose relator has length $4g$. Also, $e(\Gamma_g,X)=e(\Gamma_g,T_g)$.

Moreover, $P_T$ is obtained from $P_X$ by subdividing 
the $4g$-gon, $D$, associated to the relator in $P_X$, by mutually disjoint diagonals in $D$.
In particular, $|T| \le 6g-3$.
\end{lem}

\begin{proof}
Since $e(\Gamma_g,T)<4g-1$, by Lemma \ref{lem:surface}, 
there exists a subset $X$ in $T \cup T^{-1}$ with $|X|=2g$
such that $\form{ X }=\Gamma_g$ and 
$e(\Gamma_g,X) \le e(\Gamma_g,T)$. The argument is 
the same as in Corollary \ref{cor:min-gen}, so we omit it.

We claim that $X$ is geometric.
Indeed, since the Cayley complex for $P_T$ is planar, 
after identifying the presentation complex and $\Sigma_g$,
it gives a one-vertex cell decomposition of $\Sigma_g$.
Let $K$ be the bouquet formed by the $2g$ edges corresponding to $X$, and let $N$ be a regular neighborhood of $K$ in $\Sigma_g$. 
Since $X$ generates $\Gamma_g$, the inclusion 
$H_1(N) \to H_1(\Sigma_g)$ is surjective, 
hence 
an isomorphism since both have rank $2g$.
It implies that every  connected component of $\Sigma_g \backslash N$
is a disk, otherwise, a component would contribute 
nontrivial homology. 

Suppose there are $m$ connected components ($m$ disks). 
Then, by the Euler characteristic, 
$$1-2g+m=\chi(\Sigma_g)=2-2g,$$
hence $m=1$.

Hence $K$ is a spine of $\Sigma_g$ with one face, which is a $4g$-gon, 
$D$. 
This gives a one-relator presentation of $\Gamma_g$, $P_X$, where $X$ is the generating set, where 
the relator has length $4g$ since it corresponds to $D$. 
This presentation is geometric since the presentation complex is $\Sigma_g$.

By construction, each element of $(T\cup T^{-1}) \backslash X$ corresponds to a diagonal of $D$, and these diagonals are pairwise disjoint.  Hence $P_T$ is obtained from $P_X$
by subdividing $D$ along these diagonals. 
Since a $4g$-gon admits  at most $4g-3$ mutually disjoint 
diagonals, we have $|T| \le 2g+(4g-3)= 6g-3$.

We show  $e(\Gamma_g, X)=e(\Gamma_g, T_g).$
Indeed, the Cayley complex for $X$ is the plane, which 
is tessellated by $4g$-gons such that each vertex has $4g$ edges. 
Hence $\cay(\Gamma_g, X)$ and $\cay(\Gamma_g,T_g)$ are 
isomorphic as abstract graphs. 
It follows that  $e(\Gamma_g, X)=e(\Gamma_g, T_g).$
\end{proof}

\medskip

Los proved the following result, \cite[Theorem 1.2]{Los}.
It also follows from the above Lemma \ref{lem:geo}. 
Unlike Los \cite{Los}, our proof avoids symbolic dynamics techniques.
We remark that he also established the analogous result for 
non-orientable co-compact 
surface groups as well.

\begin{cor}[Geometric generating sets of minimal growth rate]\label{cor:geo-min}
  Let $g \ge 2$. Then, $e_{geo}(\Gamma_g)$
  is realized by $T_g$. 
\end{cor}

\begin{proof}
Let $T$ be a geometric generating set which 
realizes $e_{geo}(\Gamma_g)$. 
Since $e(\Gamma_g,T)<4g-1$, by Lemma \ref{lem:geo}, there exists a geometric generating set $X$
contained in $T \cup T^{-1}$ such that $e(\Gamma_g,X)=e(\Gamma_g,T_g)$. Since $e(\Gamma_g, X) \le e(\Gamma_g,T)=e_{geo}(\Gamma_g)$, 
we have 
$$e(\Gamma_g,T_g)=e(\Gamma_g,X)=e(\Gamma_g,T)=e_{geo}(\Gamma_g),$$
as required.
\end{proof}

\begin{rem}
There is a gap in the growth spectrum of geometric generating sets.
Indeed, suppose that $T$ is geometric and
\[
e(\Gamma_g,T)<4g-1.
\]
By Lemma~\ref{lem:geo}, there is a $2g$-element geometric generating
set $X\subset T\cup T^{-1}$ such that $P_T$ is obtained from the
$4g$-gon associated to $P_X$ by adding mutually disjoint diagonals.

Suppose that $|T|>2g$. Choose one of these diagonals, corresponding
to an element $y\in (T\cup T^{-1})\setminus X$, and set
$U=X\cup\{y\}$. The geometric presentation associated to $U$ is
obtained from the $4g$-gon by adding a single diagonal. Hence at most
one of its relations has length $3$, and the exceptional relation
$xxy$ cannot occur since $\Sigma_g$ is orientable. Therefore
Los~\cite[Lemma~5.14]{Los} applies and gives
\[
e(\Gamma_g,U)\ge 2|U|-3=4g-1.
\]
Since $U\subset T\cup T^{-1}$, this implies
\[
e(\Gamma_g,T)\ge e(\Gamma_g,U)\ge4g-1,
\]
a contradiction. Hence $|T|=2g$, and Lemma~\ref{lem:geo} gives
\[
e(\Gamma_g,T)=e(\Gamma_g,T_g)=e_{\rm geo}(\Gamma_g).
\]
Thus there is no growth rate of a geometric generating set strictly
between $e(\Gamma_g,T_g)$ and $4g-1$.

This fails for arbitrary generating sets.

For example, when $g=2$, the Nielsen transform 
\[
\{a_1, b_1, a_2, b_2\} \mapsto \{a_1a_2, b_1, a_2, b_2\}
\]
gives the following presentation of $\Gamma_2$,
\[
\form{ a_1, b_1, a_2, b_2 \mid
a_1 a_2^{-1} b_1 a_2 a_1^{-1} b_1^{-1} a_2 b_2 a_2^{-1} b_2^{-1} }.
\]
This presentation is not geometric since the relator contains $a_2$ or $a_2^{-1}$ 
more than twice. 
The corresponding growth rate is $6.9964...$ \cite{GAP4}, while
\[
e(\Gamma_2,T_2)=6.9798...<6.9964...<7.
\]
\end{rem}

\section{Questions}\label{sec:que}
The set $\xi(F_d)$ has an accumulation point (from below) $2d+1$. It follows from a small cancellation argument \cite{shukhov}.
A smaller accumulation point is known:
\[
\alpha_d(1, +\infty):=d+\sqrt{d^2+2d-3}.
\]
This value is the limit of $\alpha_d(1, q)$ as $q\to +\infty$. We point out that $\alpha_d(1, +\infty)$ is the growth rate of $\mathbb{Z}^2\ast F_{d-1}$ with the standard generating set. This can be checked by a direct computation. It also follows from a general result on the continuity of growth rate along certain convergent sequences in the space of marked groups; see \cite[Proposition 2.3]{FS}.

\begin{que}
What is the smallest accumulation point of $\xi(F_d)$ for $d\ge 2$?
\end{que}
We expect that $\alpha_d(1, +\infty)$ is the smallest accumulation point. 

It is worth pointing out that 
 $\alpha_d(1, +\infty)$ is the growth rate $e(F_d, U_\ast)$, where
\[
U_\ast:=\{a_1, a_1a_2a_1^{-1}, a_2, \dots, a_d\}.
\]
Hence, $\alpha_d(1, +\infty) \in \xi(F_d)$. 
As noted above, $2d+1$ is an accumulation point of $\xi(F_d)$. 
Observe that $5 \in \xi(F_2)$ when $d=2$; indeed, by a direct computation,
\[
e(F_2, \{a_1, a_1^2, a_2, a_2^2\})=5.
\]
However, we do not know if $2d+1 \in \xi(F_d)$ for $d>2$.

\begin{que}
Does $2d+1$ belong to $\xi(F_d)$ for every $d>2$?
    
\end{que}

We expect that $\xi(F_d)$ is not closed. 
This motivates the following question. 

\begin{que}
Is the set $\xi(F_d)$ closed in $\mathbb{R}$ for $d\ge 2$?
\end{que}

Following \cite{FS}, we denote by $d_n$ the minimal growth rate of $F_2$ with an essentially reduced generating set of cardinality $n$.
In \cite[Problem 7.5]{FS}, the authors asked for the value of $d_n$ and the generating sets that achieve it for each $n$.
As noted in Section \ref{sec:intro}, it is known that $d_2(F_2)=3$, achieved only by a free generating set $\{a_1, a_2\}$ \cite{dlHarpe2000, dlH}.
By Theorem \ref{thm:main} (1), $d_3(F_2)=1+2\sqrt{2}$, achieved by a generating set $\{a_1, a_1^2, a_2\}$.
We restate the remaining question.

\begin{que}\label{que:d4}
What are $d_n$ and the generating sets that achieve it for $n\ge 4$?
\end{que}

Although  $d_n$ is defined above only  for $F_2$, one can define the analogous quantity for $F_d$ with $d >2$ and ask the same questions.

\section*{Appendix: Ambient growth of quasi-convex subgroups of hyperbolic groups}

Let $G$ be a non-elementary word-hyperbolic group with a finite generating
set $T$.  For a subgroup $H\leq G$, put
\[
B_H(n):=H\cap B_T(n),\qquad b_H(n):=|B_H(n)|,
\]
and define the ambient growth rate
\[
e_{G,T}(H):=\limsup_{n\to\infty} b_H(n)^{1/n}.
\]

Dahmani--Futer--Wise \cite[Theorem~1.1]{DFW2019} proved that if $H$ is an infinite-index quasi-convex
subgroup, then
\[
e_{G,T}(H)<e(G,T).
\]
 Their proof uses automatic structures,
rational growth machinery, and Perron--Frobenius theory.  

We give a different argument based on a standard result which provides an
element $g$ such that
\[
\form{ H,g}=H*\form{ g}.
\]
The rest is an elementary growth-series argument, using essentially the same
free-product/$Q$-function method as in the rest of the paper.

We denote by $N(G)$ the finite radical of $G$, namely the maximal finite
normal subgroup.

\begin{prop}\label{prop:ambient-hyp}
Let $G$ be a non-elementary word-hyperbolic group with $N(G)=\{1\}$ and let
$T$ be a finite generating set.  If $H\leq G$ is an infinite-index
quasi-convex subgroup, then
\[
e_{G,T}(H)<e(G,T).
\]
\end{prop}

\begin{proof}
If $H$ is finite, then $e_{G,T}(H)=1<e(G,T)$, so assume that $H$ is infinite.

Since $H$ is quasi-convex, there exists $C\geq0$ such that for all $m,n\geq0$,
\begin{equation}\label{eq:coarse-product-app}
B_H(m+n)\subset B_H(m+C)B_H(n+C).
\end{equation}
Indeed, for $h\in B_H(m+n)$, choose a point at distance at most $m$ from
$1$ on a geodesic $[1,h]$; quasi-convexity gives a point of $H$ within
distance $C$.  Consequently,
\[
b_H(m+n)\leq b_H(m+C)b_H(n+C).
\]
Set
\[
a_n:=b_H(n+2C).
\]
Then $(a_n)$ is submultiplicative, and Fekete's lemma gives
\[
\lambda:=\lim_{n\to\infty}a_n^{1/n}=e_{G,T}(H).
\]
Moreover,
\begin{equation}\label{eq:ball-lower-app}
b_H(n+2C)=a_n\geq\lambda^n
\qquad(n\geq1).
\end{equation}

Consider the ambient spherical growth series
\[
P_{H,T}:=\sum_{h\in H}X^{|h|_T},
\]
whose radius of convergence is
\[
\rho_H=\lambda^{-1}.
\]
We claim that
\begin{equation}\label{eq:PH-div-app}
\lim_{x\uparrow\rho_H}P_{H,T}(x)=+\infty.
\end{equation}
If $\lambda>1$, then $\rho_H<1$ and, for $0\leq x<\rho_H$,
\[
P_{H,T}(x)=(1-x)\sum_{n\geq0}b_H(n)x^n.
\]
By \eqref{eq:ball-lower-app}, the ball series diverges at $\rho_H$, and
hence so does $P_{H,T}$.  If $\lambda=1$, then $\rho_H=1$, and
\eqref{eq:PH-div-app} follows since $H$ is infinite.

\bigskip

Since $N(G)=\{1\}$, by
\cite[Lemma~3.8 and Theorem~5]{Minasyan2006}, there exists
an infinite-order element $g\in G$ such that
\[
\form{ H,g}\cong H*\form{ g}.
\]
For torsion-free hyperbolic groups, the corresponding free-product extension
was proved earlier by Arzhantseva~\cite{Arzhantseva2001}.

Put $D:=|g|_T$ and
\[
\Phi(x):=x^D\bigl(P_{H,T}(x)-1\bigr).
\]
By \eqref{eq:PH-div-app}, $\Phi(0)=0$ and
$\Phi(x)\to+\infty$ as $x\uparrow\rho_H$.  Hence there exists
$x_0\in(0,\rho_H)$ such that
\[
\Phi(x_0)=1.
\]

By uniqueness of reduced normal forms in $H*\form{ g}$, all elements of the form
\[
h_1gh_2g\cdots h_kg,\qquad h_i\in H\setminus\{1\}, \quad k \ge 1,
\]
are distinct, and
\[
|h_1gh_2g\cdots h_kg|_T
   \leq \sum_{i=1}^k |h_i|_T+kD.
\]
Therefore, for every $N\geq0$,
\[
P_{G,T}(x_0)
 \geq
 \sum_{k=0}^{N}
 \left[x_0^D\bigl(P_{H,T}(x_0)-1\bigr)\right]^k
 =N+1.
\]
Thus $P_{G,T}(x_0)=+\infty$.  If $\rho_{G,T}$ is the radius of convergence
of $P_{G,T}$, then
\[
\rho_{G,T}\leq x_0<\rho_H,
\]
and hence
\[
e(G,T)=\rho_{G,T}^{-1}
   >\rho_H^{-1}
   =e_{G,T}(H).
\]
\end{proof}

\begin{rem}[The free-group case]
For the application in this paper, $G$ is a free group.  In this case,
Minasyan's result \cite{Minasyan2006} is unnecessary.  Marshall Hall's theorem says that a
finitely generated subgroup $H$ is a free factor of a finite-index subgroup
$K\leq G$.  If $H$ has infinite index in $G$, then
\[
K=H*L
\]
with $L\neq\{1\}$.  Choosing $1\neq g\in L$ gives
\[
\form{ H,g}=H*\form{ g},
\]
and the proof above applies verbatim. 
\end{rem}

\begin{rem}[The general torsion case]\label{rem:ambient-hyp}
The passage from trivial finite radical to the general case is elementary.
Put $N=N(G)$ and let
\[
\pi:G\longrightarrow \bar G:=G/N,\qquad
\bar T:=\pi(T),\qquad
\bar H:=\pi(H).
\]
Quotienting by a finite normal subgroup does not change exponential growth,
either for the group or for the ambient growth of a subgroup.  Thus
\[
e(G,T)=e(\bar G,\bar T),
\qquad
e_{G,T}(H)=e_{\bar G,\bar T}(\bar H).
\]
Moreover, $\bar H$ is quasi-convex of infinite index and
$N(\bar G)=\{1\}$.  Applying Proposition~\ref{prop:ambient-hyp} to
$(\bar G,\bar H)$ gives the full Dahmani--Futer--Wise inequality for $(G,H)$.
\end{rem}

%

%
%
%

\bibliographystyle{plain}
\bibliography{ref}

@preamble{"\def\cprime{$'$} "}

@preamble{"\def\soft#1{\leavevmode\setbox0=\hbox{h}\dimen7=\ht0\advance     \dimen7 by-1ex\relax\if t#1\relax\rlap{\raise.6\dimen7     \hbox{\kern.3ex\char'47}}#1\relax\else\if T#1\relax     \rlap{\raise.5\dimen7\hbox{\kern1.3ex\char'47}}#1\relax     \else\if d#1\relax\rlap{\raise.5\dimen7\hbox{\kern.9ex     \char'47}}#1\relax\else\if D#1\relax\rlap{\raise.5\dimen7     \hbox{\kern1.4ex\char'47}}#1\relax\else\if l#1\relax     \rlap{\raise.5\dimen7\hbox{\kern.4ex\char'47}}#1\relax     \else\if L#1\relax\rlap{\raise.5\dimen7\hbox{\kern.7ex     \char'47}}#1\relax\else\message{accent \string\soft     \space #1 not defined!}#1\relax\fi\fi\fi\fi\fi\fi} "}

@article{Arzhantseva2001,
   author  = {Arzhantseva, G. N.},
   title   = {On quasiconvex subgroups of word hyperbolic groups},
   journal = {Geom. Dedicata},
   volume  = {87},
   number  = {1--3},
   pages   = {191--208},
   year    = {2001}
}

@article{Minasyan2006,
   author  = {Minasyan, Ashot},
   title   = {On residual properties of word hyperbolic groups},
   journal = {J. Group Theory},
   volume  = {9},
   number  = {5},
   pages   = {695--714},
   year    = {2006}
 }

@unpublished{NB,
	Author = {Nyberg-Brodda, Carl-Fredrik},
	Title = {The second growth rate of free groups},
	Note = {Preprint, 2026. Personal communication.},
	Year = {},
	}

@book {Gromov,
    AUTHOR = {Gromov, Misha},
     TITLE = {Metric structures for {R}iemannian and non-{R}iemannian
              spaces},
    SERIES = {Progress in Mathematics},
    VOLUME = {152},
      NOTE = {Based on the 1981 French original,
              With appendices by M.\ Katz, P.\ Pansu and S.\ Semmes,
              Translated from the French by Sean Michael Bates},
 PUBLISHER = {Birkh\"auser Boston, Inc., Boston, MA},
      YEAR = {1999},
     PAGES = {xx+585},
      ISBN = {0-8176-3898-9},
   MRCLASS = {53C23 (53-02)},
  MRNUMBER = {1699320},
MRREVIEWER = {Igor\ Belegradek},
}

@book{KS2024,
  title={Trees of Hyperbolic Spaces},
  ISBN={9781470474256},
  ISSN={2331-7159},
  url={http://dx.doi.org/10.1090/surv/282},
  DOI={10.1090/surv/282},
  journal={Mathematical Surveys and Monographs},
  publisher={American Mathematical Society},
  author={Kapovich, Michael and Sardar, Pranab},
  year={2024}
}

@article{DFW2019,
    AUTHOR = {Dahmani, Fran{\c{c}}ois and Futer, David and Wise, Daniel T.},
     TITLE = {Growth of quasiconvex subgroups},
   JOURNAL = {Math. Proc. Cambridge Philos. Soc.},
  FJOURNAL = {Mathematical Proceedings of the Cambridge Philosophical
              Society},
    VOLUME = {167},
      YEAR = {2019},
    NUMBER = {3},
     PAGES = {505--530},
      ISSN = {0305-0041,1469-8064},
   MRCLASS = {20F67 (20E08 20F69)},
  MRNUMBER = {4015648},
MRREVIEWER = {Enric\ Ventura Capell},
       DOI = {10.1017/s0305004118000440},
       URL = {https://doi-org.kyoto-u.idm.oclc.org/10.1017/s0305004118000440},
}

@article {W,
    AUTHOR = {Wilson, John S.},
     TITLE = {On exponential growth and uniformly exponential growth for
              groups},
   JOURNAL = {Invent. Math.},
  FJOURNAL = {Inventiones Mathematicae},
    VOLUME = {155},
      YEAR = {2004},
    NUMBER = {2},
     PAGES = {287--303},
      ISSN = {0020-9910,1432-1297},
   MRCLASS = {20F69 (20E22 20F05 20F65)},
  MRNUMBER = {2031429},
MRREVIEWER = {Victor\ Petrogradsky},
       DOI = {10.1007/s00222-003-0321-8},
       URL = {https://doi-org.kyoto-u.idm.oclc.org/10.1007/s00222-003-0321-8},
}

@article {K,
    AUTHOR = {Koubi, Malik},
     TITLE = {Croissance uniforme dans les groupes hyperboliques},
   JOURNAL = {Ann. Inst. Fourier (Grenoble)},
  FJOURNAL = {Universit\'e{} de Grenoble. Annales de l'Institut Fourier},
    VOLUME = {48},
      YEAR = {1998},
    NUMBER = {5},
     PAGES = {1441--1453},
      ISSN = {0373-0956,1777-5310},
   MRCLASS = {20F32 (57M07)},
  MRNUMBER = {1662255},
MRREVIEWER = {Nadia\ Benakli},
       DOI = {10.5802/aif.1661},
       URL = {https://doi-org.kyoto-u.idm.oclc.org/10.5802/aif.1661},
}

@article{Coornaert1993,
  author  = {Coornaert, Michel},
  title   = {Mesures de {Patterson--Sullivan} sur le bord d'un espace hyperbolique au sens de {Gromov}},
  journal = {Pacific J. Math.},
  volume  = {159},
  number  = {2},
  pages   = {241--270},
  year    = {1993},
  doi     = {10.2140/pjm.1993.159.241}
}

@article {Linton, AUTHOR = {Linton, Marco}, TITLE = {Hyperbolic one-relator groups}, JOURNAL = {Canad. J. Math.}, FJOURNAL = {Canadian Journal of Mathematics. Journal Canadien de Math\'ematiques}, VOLUME = {78}, YEAR = {2026}, NUMBER = {1}, PAGES = {35--61}, ISSN = {0008-414X,1496-4279}, MRCLASS = {20F67 (20E06 20F05)}, MRNUMBER = {5020065}, DOI = {10.4153/S0008414X24000427}, URL = {https://doi.org/10.4153/S0008414X24000427}, }

@article{L,
    AUTHOR = {L\"oh, Clara},
     TITLE = {Exponential growth rates in hyperbolic groups [{\it after}
              {K}oji {F}ujiwara and {Z}lil {S}ela]},
      NOTE = {S\'eminaire Bourbaki. Vol. 2022/2023. Expos\'es 1197--1210},
   JOURNAL = {Ast\'erisque},
  FJOURNAL = {Ast\'erisque},
    NUMBER = {446},
      YEAR = {2023},
     PAGES = {Exp. No. 1206, 365--382},
      ISSN = {0303-1179,2492-5926},
      ISBN = {978-2-85629-984-5},
   MRCLASS = {20F67 (03E10 20F65)},
  MRNUMBER = {4713472},
       DOI = {10.24033/ast.1215},
       URL = {https://doi-org.kyoto-u.idm.oclc.org/10.24033/ast.1215},
}

@incollection {Pittet1993,
    AUTHOR = {Pittet, Christophe},
     TITLE = {Surface groups and quasi-convexity},
 BOOKTITLE = {Geometric group theory, {V}ol. 1 ({S}ussex, 1991)},
    SERIES = {London Math. Soc. Lecture Note Ser.},
    VOLUME = {181},
     PAGES = {169--175},
 PUBLISHER = {Cambridge Univ. Press, Cambridge},
      YEAR = {1993},
      ISBN = {0-521-43529-3},
   MRCLASS = {20F32 (57M07)},
  MRNUMBER = {1238524},
MRREVIEWER = {John\ Meier},
       DOI = {10.1017/CBO9780511661860.015},
       URL = {https://doi.org/10.1017/CBO9780511661860.015},
}

@article {Los,
    AUTHOR = {Los, J\'er\^ome},
     TITLE = {Volume entropy for surface groups via {B}owen-{S}eries-like
              maps},
   JOURNAL = {J. Topol.},
  FJOURNAL = {Journal of Topology},
    VOLUME = {7},
      YEAR = {2014},
    NUMBER = {1},
     PAGES = {120--154},
      ISSN = {1753-8416,1753-8424},
   MRCLASS = {57M07 (37B10 37B40 37E10)},
  MRNUMBER = {3180615},
MRREVIEWER = {Fr\'ed\'eric\ Palesi},
       DOI = {10.1112/jtopol/jtt032},
       URL = {https://doi.org/10.1112/jtopol/jtt032},
}

@article {Magnus,
    AUTHOR = {Magnus, Wilhelm},
     TITLE = {\"{U}ber diskontinuierliche {G}ruppen mit einer definierenden
              {R}elation. ({D}er {F}reiheitssatz)},
   JOURNAL = {J. Reine Angew. Math.},
  FJOURNAL = {Journal f\"ur die Reine und Angewandte Mathematik. [Crelle's
              Journal]},
    VOLUME = {163},
      YEAR = {1930},
     PAGES = {141--165},
      ISSN = {0075-4102,1435-5345},
   MRCLASS = {99-04},
  MRNUMBER = {1581238},
       DOI = {10.1515/crll.1930.163.141},
       URL = {https://doi-org.kyoto-u.idm.oclc.org/10.1515/crll.1930.163.141},
}

@article {dlH,
    AUTHOR = {de la Harpe, Pierre},
     TITLE = {Uniform growth in groups of exponential growth},
 BOOKTITLE = {Proceedings of the {C}onference on {G}eometric and
              {C}ombinatorial {G}roup {T}heory, {P}art {II} ({H}aifa, 2000)},
   JOURNAL = {Geom. Dedicata},
  FJOURNAL = {Geometriae Dedicata},
    VOLUME = {95},
      YEAR = {2002},
     PAGES = {1--17},
      ISSN = {0046-5755,1572-9168},
   MRCLASS = {20E07 (20F65)},
  MRNUMBER = {1950882},
MRREVIEWER = {Victor\ Petrogradsky},
       DOI = {10.1023/A:1021273024728},
       URL = {https://doi-org.kyoto-u.idm.oclc.org/10.1023/A:1021273024728},
}

@manual{GAP4,
    organization = "The GAP~Group",
    title        = "{GAP -- Groups, Algorithms, and Programming,
                    Version 4.16.0}",
    year         = 2026,
    url          = "\url{https://www.gap-system.org}",
    }

@article {GdlH,
    AUTHOR = {Grigorchuk, R. and de la Harpe, P.},
     TITLE = {On problems related to growth, entropy, and spectrum in group
              theory},
   JOURNAL = {J. Dynam. Control Systems},
  FJOURNAL = {Journal of Dynamical and Control Systems},
    VOLUME = {3},
      YEAR = {1997},
    NUMBER = {1},
     PAGES = {51--89},
      ISSN = {1079-2724,1573-8698},
   MRCLASS = {20F32 (20-02)},
  MRNUMBER = {1436550},
MRREVIEWER = {Susan\ Hermiller},
       DOI = {10.1007/BF02471762},
       URL = {https://doi-org.kyoto-u.idm.oclc.org/10.1007/BF02471762},
}

@article {shukhov,
    AUTHOR = {Shukhov, A. G.},
     TITLE = {On the dependence of the growth exponent on the length of the
              defining relation},
   JOURNAL = {Mat. Zametki},
  FJOURNAL = {Matematicheskie Zametki},
    VOLUME = {65},
      YEAR = {1999},
    NUMBER = {4},
     PAGES = {612--618},
      ISSN = {0025-567X,2305-2880},
   MRCLASS = {20F05 (20E07)},
  MRNUMBER = {1715061},
       DOI = {10.1007/BF02675367},
       URL = {https://doi-org.kyoto-u.idm.oclc.org/10.1007/BF02675367},
}

@article {S,
    AUTHOR = {Sambusetti, Andrea},
     TITLE = {Growth tightness of free and amalgamated products},
   JOURNAL = {Ann. Sci. \'Ecole Norm. Sup. (4)},
  FJOURNAL = {Annales Scientifiques de l'\'Ecole Normale Sup\'erieure.
              Quatri\`eme S\'erie},
    VOLUME = {35},
      YEAR = {2002},
    NUMBER = {4},
     PAGES = {477--488},
      ISSN = {0012-9593},
   MRCLASS = {20E06 (20F65)},
  MRNUMBER = {1981169},
MRREVIEWER = {Mihalis\ A.\ Sykiotis},
       DOI = {10.1016/S0012-9593(02)01101-1},
       URL = {https://doi-org.kyoto-u.idm.oclc.org/10.1016/S0012-9593(02)01101-1},
}

@article {FS,
    AUTHOR = {Fujiwara, Koji and Sela, Zlil},
     TITLE = {The rates of growth in a hyperbolic group},
   JOURNAL = {Invent. Math.},
  FJOURNAL = {Inventiones Mathematicae},
    VOLUME = {233},
      YEAR = {2023},
    NUMBER = {3},
     PAGES = {1427--1470},
      ISSN = {0020-9910,1432-1297},
   MRCLASS = {20F69 (20F67)},
  MRNUMBER = {4623546},
MRREVIEWER = {Sylvain\ Maillot},
       DOI = {10.1007/s00222-023-01200-w},
       URL = {https://doi-org.kyoto-u.idm.oclc.org/10.1007/s00222-023-01200-w},
}

@book{dlHarpe2000,
	address = {Chicago, IL},
	author = {de la Harpe, Pierre},
	isbn = {0-226-31719-6; 0-226-31721-8},
	mrclass = {20F65 (20F69 57M07)},
	mrnumber = {1786869 (2001i:20081)},
	mrreviewer = {Lee Mosher},
	pages = {vi+310},
	publisher = {University of Chicago Press},
	series = {Chicago Lectures in Mathematics},
	title = {Topics in geometric group theory},
	year = {2000}}

\end{document}